\documentclass{siamart0516}

\usepackage{lipsum}
\usepackage{amsfonts}
\usepackage{graphicx}
\usepackage{epstopdf}
\usepackage{algorithmic}
\ifpdf
  \DeclareGraphicsExtensions{.eps,.pdf,.png,.jpg}
\else
  \DeclareGraphicsExtensions{.eps}
\fi

\usepackage{enumerate} 

\newcommand{\TheTitle}{Weak, Strong and Linear Convergence of a Double-Layer Fixed Point Algorithm}
\newcommand{\TheAuthors}{V. I. Kolobov, S. Reich and R. Zalas}

\headers{Convergence of a Double-Layer Algorithm}{\TheAuthors}

\title{{\TheTitle}\thanks{Submitted to the editors DATE.
\funding{The second  author was partially supported by the Israel Science Foundation (Grant 389/12), the Fund for the Promotion of Research at the Technion and by the Technion General Research Fund.}}}

\author{
  Victor I. Kolobov\thanks{Department of Computer Science, The Technion - Israel Institute of Technology, 32000 Haifa, Israel (\email{kolobov.victor@gmail.com}).}
  \and
  Simeon Reich\thanks{Department of Mathematics, The Technion - Israel Institute of Technology, 32000 Haifa, Israel; (\email{sreich@tx.technion.ac.il},
    \email{rzalas@tx.technion.ac.il}).}
  \and
  Rafa\L\ Zalas\footnotemark[3]
}

\usepackage{amsopn}

\ifpdf
\hypersetup{
  pdftitle={\TheTitle},
  pdfauthor={\TheAuthors}
}
\fi

\DeclareMathOperator*{\interior}{int}
\DeclareMathOperator*{\ri}{ri}
\DeclareMathOperator*{\fix}{Fix}
\DeclareMathOperator*{\id}{Id}
\DeclareMathOperator*{\Argmax}{Argmax}
\DeclareMathOperator*{\argmax}{argmax}
\DeclareMathOperator*{\supremum}{sup}

\newsiamthm{algorithm2}{Algorithm}
\newsiamremark{remark}{Remark}
\newsiamremark{example}{Example}

\newcommand{\ac}[1] {{\color{black}#1}}

\begin{document}

\maketitle

\begin{abstract}
  In this article we consider a consistent convex feasibility problem in a real Hilbert space defined by a finite family of sets $C_i$. We are interested, in particular, in the case where for each $i$, $C_i=\fix U_i=\{z\in \mathcal H\mid p_i(z)=0\}$, $U_i\colon\mathcal H\rightarrow \mathcal H$ is a cutter and $p_i\colon\mathcal H\rightarrow [0,\infty)$ is a proximity function. Moreover, we make the following assumption: the computation of $p_i$ is at most as difficult as the evaluation of $U_i$ and this is at most as difficult as projecting onto $C_i$. We study a double-layer fixed point algorithm which applies two types of controls in every iteration step. The first one -- the outer control -- is assumed to be almost cyclic. The second one -- the inner control -- determines the most important sets from those offered by the first one. The selection is made in terms of proximity functions. The convergence results presented in this manuscript depend on the conditions which first, bind together the sets, the operators and the proximity functions and second, connect the inner and outer controls. In particular, weak regularity (demi-closedness principle), bounded regularity and bounded linear regularity imply weak, strong and linear convergence of our algorithm, respectively. The framework presented in this paper covers many known (subgradient) projection algorithms already existing in the literature; for example, those applied with (almost) cyclic, remotest-set, maximum displacement, most-violated constraint and simultaneous controls. In addition, we provide several new examples, where the double-layer approach indeed accelerates the convergence speed as we demonstrate numerically.
\end{abstract}



\section{Introduction}\label{sec:intro}
In this paper we consider a consistent \textit{convex feasibility problem} (CFP), which is to find $x\in C:=\bigcap_{i\in I} C_i\neq\emptyset$, where for each $i\in I:=\{1,\ldots,m\}$, $C_i$ is a closed and convex subset of a Hilbert space $\mathcal H$. We are interested, in particular, in the case where for every $i\in I$, there is a given cutter operator $U_i\colon \mathcal H\rightarrow\mathcal H$ and a \textit{proximity function} $p_i\colon\mathcal H\rightarrow [0,\infty)$ such that
\begin{equation}
C_i=\fix U_i=p_i^{-1}(0).
\end{equation}
We recall that an operator $U\colon\mathcal H \rightarrow \mathcal H$ with a nonempty fixed point set $\fix U:=\{z\in\mathcal H\mid Uz=z\}$ is said to be a \textit{cutter} if the inequality $\langle x-Ux, z-Ux\rangle\leq 0$ holds for all $x\in\mathcal H$ and $z\in\fix U$. Because of the operators $U_i$, oftentimes this variant of CFP is referred to as the \textit{common fixed point problem} (CFPP).

One possible approach to solving the CFP is to apply an iterative method involving metric projections onto each of the sets $C_i$, where at every step the chosen  projections $P_{C_i}$ are utilized in a certain way. This brings us to the class of the so-called \textit{projection methods}. However, in our case, the structure of the set $C_i$ provides additional information, which indicates that one \ac{can employ the cutters $U_i$ }
 instead of $P_{C_i}$ whenever its computation is less difficult. This leads us to a more general class of the so-called \textit{fixed point algorithms}.

To devise our fixed point algorithm in an effective way, we make the following two state-of-the-art assumptions: \textit{the computation of $p_i$ is at most as difficult as the evaluation of $U_i$ and this is at most as difficult as projecting onto $C_i$}. These assumptions are satisfied if, for example, the set $C_i=\{z\in\mathcal H\mid f_i(z)\leq 0\}$ is a sublevel set of a convex functional $f_i$, the operator $U_i=P_{f_i}$ is a subgradient projection and the proximity $p_i=f^+_i$. When $p_i$ is unspecified then, by default, one can use the displacement $p_i(x)=\|U_ix-x\|$ and, finally, the distance $p_i(x)=d(x,C_i)$.

Building on these assumptions, we propose a \textit{double-layer fixed point algorithm}, where for an arbitrarily given starting point $x^0\in\mathcal H$, the sequence of consecutive approximations is defined by
\begin{equation}\label{eq:xk:intro}
x^{k+1}:=x^k+\alpha_k \left (\sum_{i\in I_k}  \omega_i^k U_ix^k - x^k \right),
\end{equation}
where $I_k\subseteq J_k\subseteq I$, for each $k=0,1,2,\ldots$. The scalar $\alpha_k\in (0,2)$ is called a \textit{relaxation parameter} and $\omega_i^k\in(0,1]$ satisfy $\sum_{i\in I_k}\omega_i^k=1$. We call the subsets $I_k$ and $J_k$ an \textit{inner} and \textit{outer control}, respectively. The inner control $I_k$ determines which operators $U_j$ among all of $j\in J_k$ we want to use in order to improve the current approximation $x^k$. By default, these operators are considered to be the most significant at the step $k$, which we measure in terms of the proximity functions $p_j(\cdot)$. In particular, we may use $I_k:=\{\argmax_{j\in J_k}p_j(x^k)\}$ or even more generally, $I_k:=\{t \text{ indices from $J_k$ with the largest proximity } p_j(x^k) \}$. Following Cegielski \cite[Section 5.8.4.1]{Cegielski2012}, one could also use $I_k:=\{i\in J_k\mid p_i(x^k)\geq t \max_{j\in J_k} p_j(x^k)\}$, where $t\in [0,1]$. In view of these examples one could expect that the most substantial outer control is $J_k=I$. Indeed, this case of the outer control was confirmed both analytically and numerically to be the most efficient one in terms of the convergence rate. Nonetheless, in the case of a large number of constraints $C_i$, the cost of making a decision with $J_k=I$ may turn out to be too high. Therefore we propose to restrict this process to a subset $J_k \subseteq I$, which is not necessarily all of $I$. We emphasize here that, rather surprisingly, using $J_k\subseteq I$, we may get similar results as for $J_k=I$; see Sections \ref{sec:conv} and \ref{sec:num}. In our paper we consider an \ac{$s$-\textit{intermittent} }
 outer control, that is, a control for which
\begin{equation}\label{eq:Jk}
\ac{I= J_{k}\cup\ldots\cup J_{k+s-1}}
\end{equation}
holds true for each $k=0,1,2,\ldots$ and some $s\in \mathbb N$. In particular, this extends the cyclic control $\{J_{[k]}\}_{k=0}^\infty$ with $[k]:=(k \text{ mod }s)+1$, while assuming that $J_1\cup\ldots\cup J_s =I$.

Although in (\ref{eq:xk:intro}) we consider a quite general framework, it is not difficult to see that it covers many projection ($U_i=P_{C_i}$), subgradient projection ($U_i=P_{f_i}$) and fixed point algorithms. For example, (\ref{eq:xk:intro}) includes various sequential methods of the form  $x^{k+1}=U_{i_k}x^k$. In this case, the control $\{i_k\}_{k=0}^\infty$ can be either cyclic ($i_k=(k \text{ mod } m)+1$) or it can determine either the remotest set ($i_k=\argmax_{i\in I}d(x^k,C_i)$) or the most violated constraint for sublevel sets ($i_k=\argmax_{i\in I}f_i^+(x^k)$). In addition, (\ref{eq:xk:intro}) includes various simultaneous methods ($x^{k+1}=\sum_{i\in J_k}\omega_i^kU_ix^k$), which for $J_k\neq I$ are oftentimes called \textit{block iterative methods}. Notice that simultaneous methods cover the cyclic case by choosing a proper control $\{J_k\}_{k=0}^\infty$. In this paper we discuss how the concept of a double-layer control may indeed accelerate the convergence speed of some particular fixed point algorithms.

\subsection{Contribution}
The contribution of our paper can be summarized in three statements. All of them depend on the conditions which first, bind together the sets $C_i$, the operators $U_i$ and the proximity functions $p_i$, $i\in I$, and second, connect inner and outer controls.
\begin{itemize}
\item For the most general form of these conditions we establish the weak convergence of $\{x^k\}_{k=0}^\infty$ to some point $x^\infty \in C$, while assuming that the operators $U_i$ are weakly regular ($U_i-\id$ are demi-closed at 0), $i\in I$; see Theorem \ref{th:conv}.

\item This convergence turns out to be a strong one if, in addition, the operators $U_i$, $i\in I$, and the family $\mathcal C=\{C_i\mid i\in I\}$ are boundedly regular; see Theorem \ref{th:conv}.

\item Moreover, by restricting these conditions and by assuming that $\mathcal C$ is boundedly linearly regular, we establish a linear rate of convergence. In this case we also comment on the error bound; see Theorem \ref{th:rate}. We recall that by a \textit{linear rate of convergence} we mean that $\|x^k-x^\infty\|\leq cq^k$ for some $c\in (0, \infty)$ and $q\in(0,1)$, where $c$ and $q$ may depend on $x^0$.
\end{itemize}

In particular, all of the above-mentioned statements can be applied either with $U_i=P_{C_i}$ and $p_i(x)=d(x,C_i)$; $U_i=P_{f_i}$ and $p_i(x)=f_i^+(x)$; \ac{see Sections \ref{sec:conv:proj} and \ref{sec:conv:sub}, respectively. }
Moreover, we provide examples which easily fit the general framework.

\subsection{Related work}
We would like to begin with two important general observations regarding the types of convergence one could expect. The first is that in the infinite dimensional setting, in view of Hundal's counterexample, it may happen that even for basic cyclic or parallel projection methods the convergence can only be in the weak topology; see \cite{Hundal2004} and \cite{BauschkeMatouskovaReich2004}. Moreover, the result of Bauschke et al. \cite[Theorem 1.4]{BauschkeDeutschHundal2009} shows that norm convergence can be far away from a linear rate. Furthermore, it can be arbitrarily slow. See also Badea et al. \cite{BadeaGrivauxMuller2011} in this connection. Thus both norm and linear convergence require some additional assumptions to which we refer in general as bounded regularity and bounded linear regularity.

The germinal norm convergence result of the alternating projection method designed for two closed subspaces in Hilbert space is due to von Neumann \cite{Neumann1933}. The cyclic projection method applied to solving linear systems goes back to the seminal work of Kaczmarz \cite{Kaczmarz1937}, while its parallel version is due to Cimmino \cite{Cimmino1938}. The extension of von Neumann's result to more than two closed subspaces was established by Halperin \cite{Halperin1962}. A weak convergence theorem for a cyclic projection method for general closed and convex sets in Hilbert space was established by Bregman \cite{Bregman1965}. The case of general closed and convex sets in $\mathbb R^n$ in the context of Cimmino's method was considered by Auslender in \cite{Auslender1976}. Gurin et al. \cite{GurinPolyakRaik1967} formulated sufficient conditions for norm convergence of cyclic and remotest-set projection methods for closed and convex subsets of $\mathcal H$. These conditions are special cases of bounded regularity. Weak convergence of a simultaneous projection method in Hilbert space appeared in \cite{Pierra1984} by Pierra. Moreover, norm convergence was investigated under similar conditions to those which appeared in \cite{GurinPolyakRaik1967}. The block iterative ($J_k\subseteq I$) projection method for general closed and convex sets in $\mathbb R^n$ is due to Aharoni and Censor \cite{AharoniCensor1989}, while an extension can be found in \cite{FlamZowe1990} by Fl\r{a}m and Zowe. \ac{Some practical realizations of block iterative projection methods in the case of linear systems can be found in \cite{ElfvingHansenNikazad2016, Haltmeier2009}. }
A finite dimensional subgradient projection method with the most-violated constraint control can be found in \cite{Eremin1968} by Eremin, but also in \cite[Section 5.4.2]{CensorZenios1997} by Censor and Zenios. Cyclic and parallel subgradient projection methods in $\mathbb R^n$ were studied by Censor and Lent \cite{CensorLent1982}, and by Dos Santos \cite{Santos1987}. Both weak and strong convergence results for more general variants of simultaneous and cyclic projection methods can be found, for example, in \cite{BauschkeBorwein1996} by Bauschke and Borwein, and in \cite{Combettes1996, Combettes1997} by Combettes. Norm convergence follows there from bounded regularity of families of sets. A similar result, but for a fixed point algorithm can be found, for example, in \cite{AleynerReich2008} by Aleyner and Reich. It is worth mentioning that a very general weak convergence result for a simultaneous fixed point algorithm can be found in \cite{CegielskiCensor2011} by Cegielski and Censor.

We emphasize here that weak convergence of a special case of the double-layer projection method was established in \cite[Theorem 5.8.25]{Cegielski2012} by Cegielski, where $I_k:=\{i\in J_k\mid \omega_i^k d(x^k,C_i)\geq \delta \max_{j\in J_k} d(x^k,C_j)\}$. This is also a prototypical example of method (\ref{eq:xk:intro}).

For a more detailed overview of weak and strong convergence results we refer the reader to related monographs by Censor and Zenios \cite{CensorZenios1997}, Byrne \cite{Byrne2008}, Escalante and Raydan \cite{EscalanteRaydan2011}, Cegielski \cite{Cegielski2012}, and Popa \cite{Popa2012}. A survey of the available literature can also be found in a recent paper by Cegielski and Censor \cite{CegielskiCensor2015}.

We now concentrate on the case where the convergence becomes linear. As we have mentioned above, a linear rate may happen only under some additional assumptions. The first result of this type is due to Aronszajn \cite[\S 12]{Aronszajn1950}, who established a linear rate for the von Neumann's alternating projection method assuming that the cosine of the angle between the two closed subspaces is less than one. Whereas Agmon \cite[Theorem 3]{Agmon1954} showed this type of convergence for a projection method applied to a system of linear inequalities in $\mathbb R^n$ with the most-violated constraint control. Gurin et al. \cite[Theorem 1]{GurinPolyakRaik1967} established the same type of convergence for a projection method combined either with cyclic or the remotest-set control. This method was applied to general closed and convex sets in $\mathcal H$, while assuming that their intersection has nonempty interior or, like Agmon, that every $C_i$ is a half-space. Nonempty interior guaranteed a linear rate of convergence in the result of Pierra \cite[Theorem 1.1]{Pierra1984}, who considered a simultaneous projection method. In this spirit, Eremin \cite[Theorem, p. 142]{Eremin1968} and Polyak \cite[Theorem 6]{Polyak1969} established a linear convergence rate for a subgradient projection method with the most-violated constraint control under the Slater condition ($f_i(x)<0$ for some $x$, $i\in I$). The Slater condition also appeared in the result of De Pierro and Iusem \cite[Theorem 2]{DePierroIusem1988} in connection with the cyclic subgradient projection. Bauschke and Borwein \cite[Theorems 5.7 and 5.8]{BauschkeBorwein1996} established the same rate of convergence for a general projection algorithm involving almost cyclic and the remotest-set controls. Their main assumptions were that the family of sets $\mathcal C=\{C_i\mid i\in I\}$ is boundedly linearly regular, which includes all of the above conditions, and that the algorithm is linearly focusing ($\delta d(x^k, C_i)\leq d(x^k,C_i^k)$ and $U_i$ is the metric projection onto $C_i^k\supseteq C_i$). A closely related result for both remotest-set and parallel ($I_k=J_k=I$) projection methods can be deduced from \cite[Theorem 2.2]{BeckTeboulle2003} by Beck and Teboulle. A similar rate for a cyclic projection method, expressed in terms of bounded linear regularity, appeared in \cite[Theorem 4.5]{DeutschHundal2008} by Deutsch and Hundal. The same authors investigated convergence rates from different aspects; see \cite{DeutschHundal2006a, DeutschHundal2006b}. Recently, Bauschke et al. \cite[Theorem 6.1]{BauschkeNollPhan2015} have established a linear rate for the simultaneous ($I_k=J_k\subseteq I$) fixed point algorithm with an almost cyclic control assuming that $\mathcal C=\{C_i\mid i\in I\}$ and $U_i$'s are boundedly linearly regular. Borwein, Li and Tam \cite[Theorem 3.6]{BorweinLiTam2015} have also investigated the convergence rate of these algorithms, but in terms of H\"{o}lder bounded regularity of the operators $U_i$ and the family $\mathcal C$.

\subsection{Organization of the paper}

Our paper is organized as follows. In Section \ref{sec:preliminaries} we comment on  Fej\'er monotone sequences, quasi-nonexpansive and regular operators, as well as on regular sets. In Section \ref{sec:conv} we present two of the main results of this manuscript, namely Theorems \ref{th:conv} and \ref{th:rate}. These theorems are revisited in the context of projection and subgradient projection methods in Subsections \ref{sec:conv:proj} and \ref{sec:conv:sub}, \ac{whereas in Subsection \ref{sec:conv:loppFlagg} we explain how one can combine lopping and flagging with our algorithm. }
Finally, in Section \ref{sec:num} we provide numerical examples.

\section{Preliminaries}\label{sec:preliminaries}

Let $C\subseteq \mathcal{H}$ and $x\in \mathcal{H}$, where $\mathcal H$ is a Hilbert space. If there is a point $y\in C$ such that $\Vert y-x\Vert \leq \Vert z-x\Vert $ for all $z\in C$, then $y$ is called a \textit{metric projection} of $x$ onto $C$ and is denoted by $P_{C}x $.

Let $C$ be nonempty, closed and convex. Then for any $x\in
\mathcal{H}$, the metric projection $y:=P_{C}x$ is uniquely defined. Moreover, for every $y\in C$, we have $y=P_Cx$ if and only if
\begin{equation}
\langle x-y,z-y\rangle \leq 0
\end{equation}
for all $z\in C$; see, for example, \cite[Theorem 1.2.4]{Cegielski2012}. In addition, the functional $d(\cdot ,C):\mathcal H\rightarrow \mathbb [0,\infty)$, defined by $d(x,C):=\inf_{z\in C}\|x-z\|$, is 1-Lipschitz continuous, that is,
\begin{equation}
|d(x,C)-d(y,C)|\leq \|x-y\|
\end{equation}
for every $x,y\in\mathcal H$ and satisfies $d(x,C)=\| P_Cx-x\|$.

Let $f:\mathcal{H}\rightarrow\mathbb{R}$ be a convex and continuous function with a nonempty sublevel set $S(f,0):=\{x\mid f(x)\leq 0\}$. Denote by
$\partial f(x)$ its subdifferential, that is, $\partial f(x):=\{g\in\mathcal H\mid f(y)-f(x)\geq \langle g,y-x\rangle \text{, for all }y\in\mathcal H\}$. The continuity of $f$ implies that the set $\partial f(x)\neq\emptyset$ for all $x\in\mathcal H$ (see \cite[Proposition 16.14]{BauschkeCombettes2011}). For each $x\in\mathcal H$, let $g_f(x)\in \partial f(x)$ be a given subgradient. The so-called \textit{subgradient projection} relative to $f$ is the operator $P_{f}:\mathcal{H}\rightarrow \mathcal{H}$ defined by
\begin{equation}\label{eq:def:subProj}
P_f x:=
\begin{cases}
x-\frac{f(x)}{\| g_f(x)\|^2}g_f(x) & \text{ if $f(x)> 0$,} \\
x & \text{ otherwise;}
\end{cases}
\end{equation}
see, for example, \cite{BauschkeWangWangXu2015}. Notice that $P_f$ is well defined because \ac{$g_f(x)\neq 0$ if $f(x)>0$. }
 To simplify notation, we sometimes write
\begin{equation}\label{eq:def:subProj:short}
P_f x=x-\frac{f_+(x)}{\| g_f(x)\|^2}g_f(x),
\end{equation}
where $a_+:=\max\{0,a\}$.

By definition, $\fix P_{f}=S(f,0)$. Moreover, we have, $P_fx=P_Hx$, where $H:=\{z\in\mathcal H\mid \langle g_f(x), z-x\rangle +f(x)\leq 0\}\supseteq S(f,0)$; see \cite[Fact 2.3]{BauschkeWangWangXu2015}. Consequently, for every $x\in \mathcal H$ and $z\in S(f,0)$, we have
\begin{equation}\label{eq:subProj:cutter}
\langle x-P_fx,z-P_fx\rangle \leq 0.
\end{equation}

Finally, we recall a very useful inequality related to convex functions in $\mathbb R^n$.

\begin{lemma}\label{th:Fukushima}
Let $f_i\colon \mathbb R^n\rightarrow \mathbb R$ be convex, $i\in I$, $g(x):=\max_{i\in I}f_i(x)$ and assume that the Slater condition is satisfied, that is, $g(z)<0$ for some $z\in \mathbb R^n$. Then for every compact subset $K$, there is $\delta_K>0$ such that
\begin{equation}\label{eq:th:Fukushima}
\delta_K d(x, S(g,0))\leq g_+(x)
\end{equation}
holds for every $x\in K$.
\end{lemma}
\begin{proof}\
See, for example, \cite[Lemma 3.3]{Fukushima1984}.
\end{proof}

\subsection{Fej\'er monotone sequences}
\begin{definition}
Let $C\subseteq\mathcal H$ be a nonempty, closed and convex set, and let $\{x^k\}_{k=0}^\infty$ be a sequence in $\mathcal H$. We say that $\{x^k\}_{k=0}^\infty$ is \textit{Fej\'er monotone} with respect to $C$ if
\begin{equation}
\|x^{k+1}-z\|\leq\|x^k-z\|
\end{equation}
for all $z\in C$ and every integer $k\geq 0$.
\end{definition}

Below we present several key properties of Fej\'er monotone sequences, which we apply in our convergence analysis in order to establish weak, strong and linear convergence.

\begin{theorem} \label{th:Fejer:weak}
Let the sequence $\{x^k\}_{k=0}^\infty$ be Fej\'er monotone with respect to $C$. Then
\begin{enumerate}[(i)]
\item $\{x^k\}_{k=0}^\infty$ converges weakly to some point $x^\infty\in C$ if and only if all its weak cluster points lie in $C$;

\item $\{x^k\}_{k=0}^\infty$ converges strongly to some point $x^\infty\in C$ if and only if
    \begin{equation}
    \lim_{k\rightarrow\infty}d(x^k,C) =0;
    \end{equation}

\item if there is some constant $q\in(0,1)$ such that $d(x^{k+1},C)\leq q d(x^k,C)$ holds for every $k=0,1,2,\ldots$, then $\{x^k\}_{k=0}^\infty$ converges linearly to some point $x^\infty\in C$ and
\begin{equation}
\|x^k-x^\infty\|\leq 2d(x^0,C)q^k;
\end{equation}

\item if $\{x^{ks}\}_{k=0}^\infty$ converges linearly to some point $x^\infty\in C$, that is, $\|x^{ks}-x^\infty\|\leq c q^k$ for some constants $c>0$, $q\in(0,1)$ and integer $s$, then the entire sequence $\{x^k\}_{k=0}^\infty$ converges linearly and moreover,
\begin{equation}
\|x^k-x^\infty\|\leq \frac{c}{(\sqrt[\scriptstyle{s}]{q})^{s-1}} \left(\sqrt[\scriptstyle{s}]{q}\right)^k;
\end{equation}

\item if $\{x^{k}\}_{k=0}^\infty$ converges strongly to some point $x^\infty\in C$, then $\|x^k-x^\infty\|\leq 2d(x^k,C)$ for every $k=0,1,2,\ldots$.
\end{enumerate}
\end{theorem}
\begin{proof}
See, for example, \cite[Theorem 2.16 and Proposition 1.6]{BauschkeBorwein1996}.
\end{proof}
\subsection{Quasi-nonexpansive operators}

\begin{definition}\label{def:QNE}
Let $U:\mathcal H\rightarrow\mathcal H$ be an
operator with a fixed point, that is, $\fix U=\{z\in \mathcal H\mid z=Uz\}
\neq \emptyset$. We say that $U$ is
\begin{enumerate}[(i)]
\item \textit{quasi-nonexpansive} (QNE) if for all $x\in\mathcal H$ and all $z\in\fix U$,
\begin{equation}
\| U x -z\|\leq\| x-z\|;\label{eq:qne}%
\end{equation}

\item $\rho$\textit{-strongly quasi-nonexpansive} ($\rho$-SQNE), where $\rho\geq 0$, if for all $x\in\mathcal H$ and all $z\in\fix U$,
\begin{equation}
\| Ux-z\|^2\leq\| x-z\|^2-\rho\| U
x -x\|^2;\label{eq:sqne}
\end{equation}

\item a \textit{cutter} if for all $x\in\mathcal H$ and all $z\in\fix U$,
\begin{equation}
\langle z-U x ,x-U x \rangle\leq 0\label{eq:cutter}.
\end{equation}
\end{enumerate}
\end{definition}
A comprehensive review of the properties of QNE, SQNE and cutter operators
can be found in \cite[Chapter 2]{Cegielski2012}.

\begin{example}
Both the metric projection $P_C$ onto a nonempty, closed and convex set $C\subseteq\mathcal H$, and the subgradient projection $P_f$ associated with a continuous and convex function $f\colon\mathcal H\rightarrow \mathbb R$ with a nonempty sublevel set are cutters.
\end{example}

For a given $U\colon\mathcal H\rightarrow\mathcal H$ and $\alpha
\in (0,\infty)$, the operator $U_{\alpha}:=\id +\alpha(U-\id )$ is called an $\alpha$-\textit{relaxation of} $U$, where by $\id $ we denote the identity operator. We call $\alpha$ a \textit{relaxation parameter}. It is easy to see that for every $\alpha\neq 0$, $\fix U = \fix U_\alpha$. Usually, in connection with iterative methods, as in (\ref{eq:xk:intro}), the relaxation parameter $\alpha$ is assumed to belong to the interval $(0,2]$.

\begin{theorem}
\label{th:cuttersAndQNE}
Let $U\colon\mathcal H\rightarrow\mathcal H$ be an operator with a fixed point and let $\alpha\in(0,2]$. Then $U$ is a cutter if and only if its relaxation $\id +\alpha(U-\id )$ is $(2-\alpha)/\alpha$-strongly quasi-nonexpansive. In particular, $U$ is a cutter if and only if $U$ is $1$-strongly quasi-nonexpansive. Furthermore, $U$ is quasi-nonexpansive if and only if $\frac{1}{2}(U+\id)$ is a cutter.
\end{theorem}

\begin{proof}
\ac{See, for example, \cite[Theorem 2.1.39]{Cegielski2012}.}
\end{proof}

\begin{remark}\label{rem:SQNEandMP}
The fixed point set of a QNE operator is closed and convex (see \cite[Proposition 2.6(ii)]{BauschkeCombettes2001}).
Note that by definition and by substituting $z=P_{\fix U}x$, for any $\rho$-SQNE operator $U$ ($\rho>0$), we have
\begin{equation}
\sqrt{\rho}\| Ux-x\|\leq d(x,\fix U),
\end{equation}
an inequality which for cutters holds with $\rho=1$.
\end{remark}

\begin{theorem} \label{th:SQNE:IneqConv}
Let $U_i:\mathcal H\rightarrow \mathcal H$ be $\rho_i$-strongly quasi-nonexpansive, $\rho_i>0$, $i\in I$, with $\bigcap_{i\in I}\fix U_i\neq \emptyset$ and let $U:=\sum_{i\in I}\omega_i U_i$, where $\omega_i > 0$, $i\in I$ and $\sum_{i\in I}\omega_i=1$. Then $U$ is $\rho$-strongly quasi-nonexpansive, where $\rho:=\min_{i\in I}\rho_i>0$ and where $\fix  U =\bigcap_{i\in I}\fix U_i$. Moreover,
\begin{equation}\label{eq:th:SQNE:IneqConv:1}
\|Ux-z\|^2\leq\|x-z\|^2 - \sum_{i\in I} \omega_i \rho_i \|U_ix-x\|^2
\end{equation}
for every $x\in \mathcal H$ and $z\in \bigcap_{i\in I}\fix U_i$. Consequently,
\begin{equation}\label{eq:th:SQNE:IneqConv:2}
\frac{1}{2R}\sum_{i\in I}\omega_{i}\rho_{i}\| U_{i} x
-x\|^2\leq\| U x -x\|
\end{equation}
for any positive $R\geq\| x-z\|$.
\end{theorem}

\begin{proof}
See, for example, \cite[Theorem 2.1.50]{Cegielski2012} for the first part and \cite[Proposition 4.5]{CegielskiZalas2014} for inequalities (\ref{eq:th:SQNE:IneqConv:1}) and (\ref{eq:th:SQNE:IneqConv:2}).
\end{proof}

\subsection{Regular sets}
\begin{definition}\label{def:BR:set}
Let $C_i\subseteq\mathcal H$, $i\in I$, be closed and convex with $C:=\bigcap_{i\in I}C_i \neq\emptyset$ and let $\mathcal C =\{C_i\mid i\in I\}$. Let $S\subseteq\mathcal H$ be nonempty. We say that $\mathcal C$ is
\begin{enumerate}[(i)]
\item \textit{regular} over $S$ if for every sequence $\{x^k\}_{k=0}^\infty\subseteq S$,
\begin{equation}\label{eq:def:BR:set}
\lim_{k\rightarrow\infty} \max_{i\in I} d(x^k, C_i) =0 \Longrightarrow
\lim_{k\rightarrow\infty} d(x^k, C)=0;
\end{equation}

\item \textit{linearly regular} over $S$ if there is $\kappa_S>0$ such that for every $x\in S$,
\begin{equation}\label{eq:def:BLR:set}
d(x, C)\leq\kappa_S \max_{i\in I} d(x, C_i).
\end{equation}
\end{enumerate}
If any of the above regularity conditions holds for every subset $S\subseteq\mathcal H$, then we simply omit the phrase ``over $S$". If the same holds, but restricted to bounded subsets $S\subseteq\mathcal H$, then we precede the term with the adverb \textit{boundedly}.
\end{definition}
Below we present some sufficient conditions for regularities of sets. Many more conditions can be found, for example, in \cite[Fact 5.8]{BauschkeNollPhan2015}.
\begin{theorem}\label{eq:thm:BLR:prop}
Let $C_i\subseteq\mathcal H$, $i\in I$, be closed and convex with $C:=\bigcap_{i\in I}C_i \neq\emptyset$ and let $\mathcal C =\{C_i\mid i\in I\}$. Then the following statements hold:
\begin{enumerate}[(i)]
\item if $\dim\mathcal H<\infty$, then $\mathcal C$ is boundedly regular;
\item if $C_j\cap\interior\bigcap_{i\in I\setminus\{j\}}C_i\neq\emptyset$, then $\mathcal C$ is boundedly linearly regular;
\item if each $C_i$ is a half-space, then $\mathcal C$ is linearly regular;
\item if $\dim\mathcal H<\infty$, $C_j$ is a half-space, $j\in J\subseteq I$, and $\bigcap_{j\in J}C_j \cap \bigcap_{i\in I\setminus J}\ri C_i\neq\emptyset$, then $\mathcal C$ is boundedly linearly regular.

\end{enumerate}
\end{theorem}

\begin{remark}
Following \cite{BadeaGrivauxMuller2011} by Badea, Grivaux and M\"uller, we recall the notion of the extended cosine of the Friedrichs angle corresponding to a finite family of closed subspaces $C_i\subseteq\mathcal H$, $i\in I$, for which $C:=\bigcap_{i\in I}C_i \neq\emptyset$, that is,
\begin{equation}\label{eq:def:Friedrich}
c(C_1,...,C_m):=\sup\left\{\frac{1}{m-1}\frac{\sum_{i\neq j}\langle c_i,c_j\rangle }{\sum_{i=1}^m\|c_i\|^2}\mid c_i\in C_i\cap C^{\perp}\text{ and } \sum_{i=1}^m\|c_i\|^2\neq 0\right\}.
\end{equation}
In order to relate this parameter to the linear regularity constant, we focus our attention on the smallest possible $\kappa_{\mathcal H}$ which appears in (\ref{eq:def:BLR:set}), that is,
\begin{equation}
  \kappa(C_1,...,C_m):=\supremum_{x\notin C}\frac{d(x,C)}{\max_{i\in I}d(x,C_i)}.
\end{equation}
The inverse of this constant is oftentimes called the \textit{inclination} of the subspaces; see, for example, \cite{PustylnikReichZaslavski2012, PustylnikReichZaslavski2013} by Pustylnik, Reich and Zaslavski. By \cite[Proposition 3.9]{BadeaGrivauxMuller2011}, we have
\begin{equation}\label{eq:thm:incBLR:res}
\frac{1}{\sqrt{(1-c)(m-1)}}\leq\kappa\leq\frac{2m}{(1-c)(m-1)}.
\end{equation}
In particular, $\kappa<\infty$ if and only if $c<1$.
\end{remark}

\subsection{Regular operators}
We extend the notion of (boundedly) (linearly) regular families of sets to (boundedly) (linearly) regular operators. We adopt the naming convention in \cite{BauschkeNollPhan2015}. To make our nomenclature consistent, we also introduce the term ``weakly regular''.

\begin{definition}\label{def:BR:oper}
Let $U:\mathcal H\rightarrow\mathcal H$ be an
operator with a fixed point, that is, $\fix U
\neq \emptyset$ and let $S\subseteq\mathcal H$ be nonempty. We say that the operator $U$ is
\begin{enumerate}[(i)]
\item \textit{weakly regular} over $S$ if for any sequence $\{x^{k}\}_{k=0}^\infty \subseteq S$ and $x^\infty\in \mathcal H$,
\begin{equation}
\label{eq:def:DC}
\left .
\begin{array}{l}
x^{k}\rightharpoonup x^\infty\\
U x^k-x^k\rightarrow 0
\end{array}
\right\}\quad\Longrightarrow\quad x^\infty\in \fix U;
\end{equation}

\item \textit{regular} over $S$ if for any sequence $\{x^{k}\}_{k=0}^\infty \subseteq S$,
\begin{equation} \label{eq:def:BR:oper}
\lim_{k\rightarrow\infty}\| Ux^k-x^k\| =0\quad\Longrightarrow\quad \lim_{k\rightarrow\infty}d(x^k,\fix U)=0;
\end{equation}

\item \textit{linearly regular} over $S$ if there is $\delta_S>0$ such that for every $x\in S$,
\begin{equation} \label{eq:def:BLR:oper}
\delta_S d(x,\fix U)\leq \|Ux-x\|.
\end{equation}
\end{enumerate}
If any of the above regularity conditions holds for every subset $S\subseteq\mathcal H$, then we simply omit the phrase ``over $S$". If the same condition holds when restricted to bounded subsets $S\subseteq\mathcal H$, then we precede the term with the adverb \textit{boundedly}. Since there is no need to distinguish between boundedly weakly and weakly regular operators, we call both weakly regular.

\end{definition}
\begin{remark}
The above definition, points (ii) and (iii), indeed extends Definition \ref{def:BR:set}. To see this, let $C_i\subseteq\mathcal H$, $i\in I$, be closed and convex with $C:=\bigcap_{i\in I}C_i \neq\emptyset$ and let $\mathcal C =\{C_i\mid i\in I\}$. Moreover, let $U:=P_{C_{i(x)}}x$, where $i(x):=\min \{i\in I \mid d(x,C_i)= \max_{j\in I}d(x,C_j)\}.$ It is not difficult to see that $\fix U=C$ and  $\|Ux-x\|=\max_{i\in I}d(x,C_i)$. Therefore the family $\mathcal C$ is (boundedly) (linearly) regular if and only if the operator $U$ is (boundedly) (linearly) regular.
\end{remark}

Note that for $S=\mathcal H$, saying that $U$ is weakly regular is nothing else than saying that $U-\id$ is demi-closed at 0. The definition of demi-closed operators goes back to papers by Browder and Petryshyn \cite{BrowderPetryshyn1966}, and by Opial \cite{Opial1967}. Some authors refer to condition (\ref{eq:def:DC}) as a \textit{demi-closedness principle}; \ac{see \cite[Definition 4.2]{Cegielski2015} and \cite[page 388]{Cegielski2015b}. }
 The concept of a demi-closed operator has been extended in many directions. For example Bauschke, Chen and Wang \cite{BauschkeChenWang2014} have recently considered a \textit{fixed point closed mapping} (defined in \cite[Lemma 2.1]{BauschkeChenWang2014}), where the weak convergence to $x^\infty$ was replaced by a strong one. On the other hand, Cegielski \cite[Definition 4.4]{Cegielski2015} has considered a demi-closedness condition referring to a family of operators instead of to a single one; in this direction see also \cite[Lemma 3.4]{ReichZalas2016}.

A prototypical version of condition (\ref{eq:def:BR:oper}) with a continuous $U$ can be found in \cite[Theorem 1.2]{PetryshynWilliamson1973} by Petryshyn and Williamson. To the best of our knowledge, boundedly regular operators in this form first appeared in \cite{CegielskiZalas2013} by Cegielski and Zalas, who applied them to variational inequalities under the name \textit{approximately shrinking}. Independently, Bauschke, Noll and Phan \cite{BauschkeNollPhan2015} investigated unrestricted iterations of these operators in connection with common fixed point problems under the name \textit{boundedly regular}. Many properties of these operators (under the name approximately shrinking) were presented in \cite{CegielskiZalas2014} with some extensions in \cite{Zalas2014}, \cite{ReichZalas2016} and \cite{Cegielski2016}, and further applications in \cite{Cegielski2015} and \cite{CegielskiMusallam2016}.

The name \textit{linearly boundedly regular} operators was proposed by Bauschke, Noll and Phan, who applied them to establish a linear rate of convergence for a block iterative fixed point algorithm \cite[Theorem 6.1]{BauschkeNollPhan2015}. To the best of our knowledge, the concept of this type of operator goes back to Outlaw \cite[Theorem 2]{Outlaw1969}, and Petryshyn and Williamson \cite[Corollary 2.2]{PetryshynWilliamson1973}. A closely related condition called a \textit{linearly focusing} algorithm was studied by Bauschke and Borwein \cite[Definition 4.8]{BauschkeBorwein1996}. The concept of a focusing algorithm goes back to Fl\r{a}m and Zowe \cite[Section 2]{FlamZowe1990}, and can also be found in \cite[Definition 1.2]{Combettes1997} by Combettes. Linearly regular operators appeared in \cite[Definition 3.3]{CegielskiZalas2014} by Cegielski and Zalas as \textit{linearly shrinking} operators.

\begin{proposition}\label{th:DCBRLBR}
Let $U:\mathcal{H}\rightarrow \mathcal{H}$ be such that $\fix U$ is nonempty. Then the operator
\begin{enumerate}[(i)]
\item $U$ is boundedly regular whenever it is linearly boundedly regular.
\end{enumerate}
Moreover, if $\fix U$ is closed and convex, for example, when $U$ is quasi-nonexpansive, then
\begin{enumerate}[(i)]
\setcounter{enumi}{1}
\item $U$ is weakly regular whenever $U$ is boundedly regular;

\item $U$ is boundedly regular whenever $U$ is weakly regular and $\dim \mathcal{H}<\infty$.
\end{enumerate}
\end{proposition}
\begin{proof}
(i) follows directly from Definition \ref{def:BR:oper}. \ac{Parts (ii) and (iii) follow from the arguments used to prove \cite[Proposition 4.1]{CegielskiZalas2014}.}
\end{proof}
It was shown by Opial \cite[Lemma 2]{Opial1967} that a nonexpansive (1-Lipschitz) mapping satisfies condition (\ref{eq:def:DC}) with $S=\mathcal H$. Therefore for $\mathcal H=\mathbb R^n$ a nonexpansive mapping is boundedly regular whenever it has a fixed point. It is worth mentioning that in a general Hilbert space the weak regularity of $U$ is only a necessary condition for implication (\ref{eq:def:BR:oper}) and even a firmly nonexpansive mapping may not have this property; see either \cite[Example 2.9]{Zalas2014} or \cite[Examples 2.17 and 2.18]{ReichZalas2016}.

We conclude this section with the following example.

\begin{example}\label{ex:subProj}
Let $f\colon\mathcal H\rightarrow \mathbb R$ be convex and continuous.
\begin{enumerate}[(a)]
  \item If $f$ is Lipschitz continuous on bounded sets, then $P_f$ is weakly regular; see, for example \cite[Theorem 4.2.7]{Cegielski2012}. We recall that $f$ is Lipschitz continuous on bounded sets $\Longleftrightarrow$ $f$ maps bounded sets onto bounded sets $\Longleftrightarrow$ $\partial f$ is uniformly bounded on bounded sets; see, for example, \cite[Proposition 7.8]{BauschkeBorwein1996}. All three conditions hold true if $\mathcal H=\mathbb R^n$.

  \item If $\mathcal H=\mathbb R^n$, then by (a) and Proposition \ref{th:DCBRLBR}, $P_f$ is boundedly regular. See also \cite[Lemma 24]{CegielskiZalas2013}.

  \item If $f(z)<0$ for some $z\in\mathbb R^n$, then $P_f$ is boundedly linearly regular. Indeed, by (a) and by Lemma \ref{th:Fukushima}, for every compact $K\subseteq\mathbb R^n$, there are $\delta_K, \Delta_K>0$ such that $\|\partial f(x)\|\leq \Delta_K$ and $\delta_K d(x, S(f,0))\leq f_+(x)$ for every $x\in K$. Therefore we can write
      \begin{equation}
      \|x-P_fx\| = \frac{f_+(x)}{\|g_f(x)\|}\geq \frac{\delta_K}{\Delta_K}d(x, S(f,0)).
      \end{equation}
\end{enumerate}
\end{example}

\section{Weak, strong and linear convergence results}\label{sec:conv}
Unless otherwise stated, we assume from now on that $C:=\bigcap_{i\in I}C_i \neq\emptyset$ and that for every $i\in I$, we have
\begin{equation}
C_i=\fix U_i=p_i^{-1}(0)\subseteq\mathcal H,
\end{equation}
where $U_i\colon\mathcal H\rightarrow\mathcal H$ is a cutter and $p_i\colon\mathcal H\rightarrow [0,\infty)$, $i\in I$.

\begin{theorem}[Weak and strong convergence]\label{th:conv}
Let the sequence $\{x^k\}_{k=0}^\infty$ be defined by the double-layer fixed point algorithm, that is,
\begin{equation}\label{eq:th:conv:xk}
x^0\in\mathcal H;\quad x^{k+1}:=x^k+\alpha_k \left (\sum_{i\in I_k}  \ac{\omega_i^k}
 U_ix^k - x^k \right) \text{ for } k=0,1,2,\ldots,
\end{equation}
where $I_k\subseteq J_k\subseteq I$, $\alpha_k\in (0,2)$ and $\omega_i^k\in  (0,1]$ are such that $\sum_{i\in I_k}\omega_i^k=1$.
Assume that
\begin{enumerate}[(i)]
\item $\alpha_k\in[\alpha^-,\alpha^+]$ and $\omega^i_k\in[\omega^-, 1]$, where $\alpha^-, \omega^- \in (0,1]$ and $\alpha^+\in [1,2)$;

\item \ac{$I= J_k\cup\ldots\cup J_{k+s-1}$ }
 for every $k=0,1,2,\ldots$ and some integer \ac{$s\geq 1$;}

\item the inner and outer controls satisfy
\begin{equation}
\lim_{k\rightarrow\infty}\max_{i\in I_k}p_i(x^k)=0 \quad\Longrightarrow\quad \lim_{k\rightarrow\infty}\max_{j\in J_k}p_j(x^k)=0;\label{eq:th:conv:pipj}
\end{equation}

\item for every bounded sequence $\{y^k\}_{k=0}^\infty\subseteq\mathcal H$ and $i\in I$,
\begin{equation}
\lim_{k\rightarrow\infty}p_i(y^k)= 0 \quad\Longleftrightarrow \quad\lim_{k\rightarrow\infty}\|U_iy^k-y^k\|= 0.\label{eq:th:conv:piUi}
\end{equation}
\end{enumerate}

If  for each $i\in I$, the operator $U_i$ is weakly regular, then the sequence $\{x^k\}_{k=0}^\infty$ converges weakly to some point $x^\infty \in C$.
Moreover, if for each $i\in I$, the operator $U_i$ is boundedly regular and the family $\mathcal C:=\{\fix U_i\mid i\in I\}$ is boundedly regular, then the convergence to $x^\infty$ is in norm.
\end{theorem}

\begin{proof}
To show both weak and strong convergence, we apply Theorem \ref{th:Fejer:weak} (i) and (ii), respectively.

We first show that $\{x^k\}_{k=0}^\infty$ is Fej\'er monotone with respect to $C$. Indeed, for every $k=0,1,2,\ldots,$ let
\begin{equation} \label{eq:th:conv:proof:Tk}
T_k:=\sum_{i\in i_k}\omega_i^k\left(\id+\alpha_k(U_i-\id)\right).
\end{equation}
We can write $x^{k+1}=T_kx^k$. It is not difficult to see that by Theorems \ref{th:cuttersAndQNE} and \ref{th:SQNE:IneqConv}, $T_k$ is $\rho_k$-SQNE and $C\subseteq\fix T_k=\bigcap_{i\in I_k}\fix U_i$. Moreover, we have
\begin{equation} \label{eq:th:conv:proof:rho}
\rho_k=\frac{2-\alpha_k}{\alpha_k}\geq \frac{2-\alpha^+}{\alpha^+}:=\rho>0.
\end{equation}
In addition, for every $z\in C$, we get
\begin{equation} \label{eq:th:conv:proof:1}
\|x^{k+1}-z\|^2=\|T_kx^k-z\|^2\leq\|x^k-z\|^2-\rho \|T_k x^k-x^k\|^2 \leq \|x^k-z\|^2,
\end{equation}
which immediately implies that $\{x^k\}_{k=0}^\infty$ is Fej\'er monotone with respect to $C$.

Next we show that
\begin{equation} \label{eq:th:conv:proof:step2}
\lim_{k\rightarrow\infty} \max_{j\in J_k}\|U_jx^k-x^k\|=0.
\end{equation}
Note that since $\{x^k\}_{k=0}^\infty$ is Fej\'er monotone, it is bounded. Moreover, since the sequence $\{\|x^k-z\|\}_{k=0}^\infty$ is decreasing and bounded, it is convergent as well \ac{and, by (\ref{eq:th:conv:proof:rho})--(\ref{eq:th:conv:proof:1}),}
\begin{equation}\label{eq:th:conv:proof:2}
\lim_{k\rightarrow\infty}\|T_kx^k-x^k\|=0,
\end{equation}
which is equivalent to
\begin{equation}\label{eq:th:conv:proof:3}
\lim_{k\rightarrow\infty}\|x^{k+1}-x^k\|=0.
\end{equation}
Let $z\in C$ and let $R>0$ be such that $\|x^k-z\|\leq R$. Using Theorem \ref{th:SQNE:IneqConv}, we get
\begin{align}\label{eq:th:conv:proof:4}
\|T_kx^k-x^k\|&\geq \frac{1}{2R}\sum_{i\in I_k}\omega^-\rho\|U_ix^k-x^k\|^2,
\end{align}
which together with (\ref{eq:th:conv:proof:2}) yields
\begin{equation} \label{eq:th:conv:proof:5}
\lim_{k\rightarrow\infty} \max_{i\in I_k}\|U_ix^k-x^k\|=0.
\end{equation}
Therefore, by (\ref{eq:th:conv:pipj}) and (\ref{eq:th:conv:piUi}), we see that
(\ref{eq:th:conv:proof:step2}) holds true, as asserted.

\textbf{Case 1.} Now assume that the operators $U_i$ are weakly regular, $i\in I$. Let $x$ be an arbitrary weak cluster point of $\{x^k\}_{k=0}^\infty$ and let $x^{n_k}\rightharpoonup x$. Let $i\in I$ be fixed and for every $k=0,1,2,\ldots$, let $r_k\in\{0,1,2,\ldots,s-1\}$ be the smallest number such that $i\in J_{n_k+r_k}$. By (\ref{eq:th:conv:proof:3}), $x^{n_k+r_k}\rightharpoonup x$ and by (\ref{eq:th:conv:proof:step2}),
\begin{equation}\label{eq:th:conv:proof:case1:1}
\lim_{k\rightarrow\infty} \|U_ix^{n_k+r_k}-x^{n_k+r_k}\|=0.
\end{equation}
Hence the weak regularity of $U_i$ yields that $x\in \fix U_i$ and the arbitrariness of $i$ implies that $x\in C$. Thus we have shown that each weak cluster point of $\{x^k\}_{k=0}^\infty$ lies in $C$, which, by Theorem \ref{th:Fejer:weak} (i), means that $\{x^k\}_{k=0}^\infty$ converges weakly to some point in $C$.

\textbf{Case 2.} Next we assume that the operators $U_i$ are boundedly regular, $i\in I$, and the family $\mathcal C$ is boundedly regular. By (\ref{eq:th:conv:proof:step2}),
\begin{equation} \label{eq:th:conv:proof:case2:1}
\lim_{k\rightarrow\infty} \max_{j\in J_k} d(x^k, C_j)=0.
\end{equation}
By Theorem \ref{th:Fejer:weak} (ii) and by the assumed bounded \ac{regularity of $\mathcal{C}$, }
it suffices to show that
\begin{equation} \label{eq:th:conv:proof:case2:2}
\lim_{k\rightarrow\infty} \max_{i\in I} d(x^k, C_i)=0.
\end{equation}
Indeed, let $i\in I$ and let $r_k\in\{0,1,2,\ldots,s-1\}$ be the smallest integer such that $i\in J_{k+r_k}$. By (\ref{eq:th:conv:proof:case2:1}),
\begin{equation} \label{eq:th:conv:proof:case2:3}
\lim_{k\rightarrow\infty} d(x^{k+r_k}, C_i)=0.
\end{equation}
Moreover, by the triangle inequality and by the definition of the metric projection, we get
\begin{align}\label{eq:th:conv:proof:case2:4}
d(x^k, C_i)& = \|x^k-P_{C_i}x^k\|\leq \|x^k-P_{C_i}x^{k+r_k}\| \nonumber\\
& \leq \|x^k-x^{k+r_k}\|+\|x^{k+r_k}-P_{C_i}x^{k+r_k}\|,
\end{align}
which, when combined with (\ref{eq:th:conv:proof:3}) and (\ref{eq:th:conv:proof:case2:3}), yields
\begin{equation}\label{eq:th:conv:proof:case2:5}
\lim_{k\rightarrow\infty}d(x^k, C_i)=0.
\end{equation}
The arbitrariness of $i\in I$ implies (\ref{eq:th:conv:proof:case2:2}) which, in view of Theorem \ref{th:Fejer:weak} (ii), completes the proof.
\end{proof}

\begin{theorem}[Linear convergence]\label{th:rate}
Let the sequence $\{x^k\}_{k=0}^\infty$ be defined as in Theorem \ref{th:conv} under conditions (i) and (ii). Moreover, assume that

\begin{enumerate}[(i)]
\setcounter{enumi}{2}
\item the inner and outer controls satisfy for each $k=0,1,2,\ldots$,
\begin{equation}\label{eq:th:rate:IkJk}
I_k\cap \Argmax_{j\in J_k} p_j(x^k) \neq \emptyset;
\end{equation}

\item for $r=d(x^0,C)>0$, there are numbers $\delta_r, \Delta_r>0$ such that for every $x\in B(P_Cx^0,r)$ and $i\in I$,
\begin{equation}\label{eq:th:rate:dipiUi}
\delta_r d(x,C_i)\leq p_i(x) \leq \Delta_r \|U_ix-x\|.
\end{equation}
\end{enumerate}
If the family $\mathcal C:=\{\fix U_i\mid i\in I\}$ is $\kappa_r$-boundedly linearly  regular over $B(P_Cx^0,r)$, then the sequence $\{x^k\}_{k=0}^\infty$ converges linearly to some point $x^\infty\in C$, that is, $\|x^k-x^\infty\|\leq c_rq_r^k$, where
\begin{equation}\label{eq:th:rate:estimate}
c_r=\frac{2d(x^0,C)}
{ q_r^{s-1}}
\qquad\text{and}\qquad
q_r=\sqrt[\scriptstyle{2s}]{1-
\frac{\omega^-(2-\alpha^+)(\alpha^-)^2}{s \alpha^+}\cdot
\left(\frac{\delta_r}{\kappa_r\Delta_r }\right)^2}.
\end{equation}
\emph{(Error bound)} Moreover, we have the following estimate:
\begin{equation} \label{eq:th:rate:estimateEB}
\frac{\delta_r}{2\kappa_r}\|x^k-x^\infty\|\leq \max_{i\in I} p_i(x^k)\leq \Delta_r c_r q_r^k.
\end{equation}
\end{theorem}

\begin{proof}
For every $k=0,1,2,\ldots$, let $T_k$ be defined as in the proof of Theorem \ref{th:conv}, that is,
\begin{equation} \label{eq:th:rate:proof:Tk}
T_k=\sum_{i\in I_k}\omega_i^k\left(\id+\alpha_k(U_i-\id)\right).
\end{equation}
As before, we can write that $x^{k+1}=T_kx^k$ knowing, in addition, that $T_k$ is $\rho_k$-SQNE with $\fix T_k=\bigcap_{i\in I_k}\fix U_i$. Moreover, we can estimate
\begin{equation}\label{eq:th:rate:proof:rho}
\rho^-:=\frac{2-\alpha^+}{\alpha^+}\leq\rho_k \leq \frac{2-\alpha^-}{\alpha^-}:=\rho^+.
\end{equation}
Furthermore, we see that $\id+\alpha_k(U_i-\id)$ is $\rho_k$-SQNE as well. We divide the rest of the proof into three steps.

\textbf{Step 1.} First, we show that for every $\nu\in I$, $z\in C$, and $k=0,1,2,\ldots,$ we have
\begin{equation}\label{eq:th:rate:proof:step1}
d^{2}(x^{ks},C_\nu)\leq \frac {s} {\omega^-\rho^-(\alpha^-)^2} \left(\frac{\Delta_r }{\delta_r}\right)^2\left(\| x^{ks}-z\|^{2}-\| x^{(k+1)s}-z\|^{2}\right),
\end{equation}
Indeed, let $\nu \in I$, $z\in C$ and for every $k=0,1,2,\dots,$ let $m_k\in\{ks,\ldots,(k+1)s-1\}$ be the smallest integer such that $\nu\in J_{m_k}$. Define
\begin{equation}\label{eq:th:rate:proof:ikdef}
i_k:=\argmax_{j\in I_{m_k}} p_j (x^{m_k}),
\end{equation}
which, by (\ref{eq:th:rate:IkJk}), satisfies
\begin{equation}\label{eq:th:rate:proof:ik}
i_k\in\Argmax_{j\in J_{m_k}} p_j (x^{m_k}).
\end{equation}
Since the function $d(\cdot,C_\nu)$ is nonexpansive, we have
\begin{equation}
d(x^{ks},C_\nu)-d(x^{m_k},C_\nu)\leq |d(x^{ks},C_\nu)-d(x^{m_k},C_\nu)|  \leq \| x^{ks}-x^{m_k}\|
\end{equation}
and therefore
\begin{equation}\label{eq:th:rate:proof:1}
d(x^{ks},C_\nu) \leq \| x^{ks}-x^{m_k}\| + d(x^{m_k},C_\nu)\leq \sum_{n=ks}^{m_k-1}\|x^{n+1}-x^n\| + d(x^{m_k},C_\nu),
\end{equation}
where the last step follows from the triangle inequality.
For $r=d(x^0,C)$, let $\delta_r$ and $\Delta_r$ be as in (\ref{eq:th:rate:dipiUi}). Note that since $\{x^k\}_{k=0}^\infty$ is Fej\'er monotone, $\{x^k\}_{k=0}^\infty\subseteq B(P_Cx^0,r)$ and consequently, for every $i\in I$ and $k=0,1,2,\ldots,$
\begin{equation}\label{eq:th:rate:proof:dipiuixk}
\delta_r d(x^k,C_i)\leq p_i(x^k) \leq \Delta_r \|U_ix^k-x^k\|.
\end{equation}
By the definition of $i_k$ (\ref{eq:th:rate:proof:ik}) and by (\ref{eq:th:rate:proof:dipiuixk}), we get
\begin{equation}\label{eq:th:rate:proof:nuik}
d(x^{m_k},C_\nu)\leq \frac{1}{\delta_r}p_\nu(x^{m_k})\leq \frac{1}{\delta_r}p_{i_k}(x^{m_k})\leq\frac{\Delta_r}{\delta_r}\|U_{i_k}x^{m_k}-x^{m_k}\|.
\end{equation}
Now, by combining (\ref{eq:th:rate:proof:1}), \ac{the Cauchy-Schwarz inequality for real numbers }
and (\ref{eq:th:rate:proof:nuik}),
\begin{align}\label{eq:th:rate:proof:2}
d^2(x^{ks},C_\nu) &\leq \left(\sum_{n=ks}^{m_k-1}\|x^n-x^{n+1}\| + d(x^{m_k},C_\nu)\right)^2\nonumber \\
&\leq (m_k-ks+1)\left(\sum_{n=ks}^{m_k-1}\|x^n-x^{n+1}\|^2 + d^2(x^{m_k},C_\nu) \right) \nonumber\\
&\leq s\left(\frac{\Delta_r}{\delta_r}\right)^2 \left(\sum_{n=ks}^{m_k-1}\|x^n-x^{n+1}\|^2 +  \|U_{i_k} x^{m_k}-x^{m_k}\|^2 \right),
\end{align}
where the last inequality follows from the fact that $\Delta_r\geq \delta_r$. Indeed, since $U_i$ is a cutter, using (\ref{eq:th:rate:dipiUi}), we get for every $x\in B(P_Cx^0,r)$ that $\|U_ix-x\|\leq d(x,C_i)\leq \Delta_r/\delta_r \|U_ix-x\|$.

We now estimate the first term in (\ref{eq:th:rate:proof:2}). By the $\rho_k$-strong quasi-nonexpansivity of $T_k$ and (\ref{eq:th:rate:proof:rho}), we get for all $n=0,1,2,\ldots,$
\begin{equation}\label{eq:th:rate:sqne}
\|x^n-x^{n+1}\|^2\leq\frac{1}{\rho^-}\left(\|x^n-z\|^2-\|x^{n+1}-z\|^2\right).
\end{equation}
Substituting (\ref{eq:th:rate:sqne}) into the first term produces the following estimates:
\begin{align}\label{eq:th:rate:proof:3}
\sum_{n=ks}^{m_k-1}\|x^n-x^{n+1}\|^2&\leq\sum_{n=ks}^{m_k-1}\frac{1}{\rho^-}\left(\|x^n-z\|^2-\|x^{n+1}-z\|^2\right)\nonumber \\
&=\frac{1}{\rho^-}\left(\|x^{ks}-z\|^2-\|x^{m_k}-z\|^2\right)\nonumber \\
&\leq \frac{1}{\omega^- \rho^- (\alpha^-)^2}\left(\|x^{ks}-z\|^2-\|x^{m_k}-z\|^2\right).
\end{align}
Next, we estimate the second summand in (\ref{eq:th:rate:proof:2}). Recall that, by definition, $x^{m_k+1}= T_{m_k} x^{m_k}$. Therefore, by Theorem \ref{th:SQNE:IneqConv} applied to $x=x^{m_k}$ and $U=T_k$, and by the Fej\'er monotonicity of $\{x^k\}_{k=0}^\infty$, we have
\begin{align}\label{eq:th:rate:proof:4}
\omega^-\rho^-(\alpha^-)^2\|U_{i_{m_k}} x^{m_k}-x^{m_k}\|^2 & \leq
\| x^{m_k}-z\|^2 - \|T_{m_k}x^{m_k} -z \|^2 \nonumber\\
&= \| x^{m_{k}}-z\|^{2}-\| x^{m_{k}+1}-z\|^{2} \nonumber\\
&\leq \| x^{m_k}-z\|^{2}-\| x^{(k+1)s}-z\|^{2} .
\end{align}
Combining (\ref{eq:th:rate:proof:3}) and (\ref{eq:th:rate:proof:4}) with (\ref{eq:th:rate:proof:2}), we derive (\ref{eq:th:rate:proof:step1}), as asserted.

\textbf{Step 2.} Now we show that \ac{the estimate}
\begin{equation}\label{eq:th:rate:proof:estimate}
\|x^k-x^\infty\|\leq c_rq_r^k
\end{equation}
holds true. Indeed, by setting $z=P_{C}x^{ks}$ in (\ref{eq:th:rate:proof:step1}) and by noting that
\begin{equation}
\| x^{(k+1)s}-P_{C}x^{(k+1)s}\|\leq\| x^{(k+1)s}-P_{C}x^{ks}\|,
\end{equation}
we get
\begin{equation}\label{eq:th:rate:proof:5}
d^{2}(x^{ks},C_\nu)\leq \frac {s} {\omega^-\rho^-(\alpha^-)^2} \left(\frac{\Delta_r }{\delta_r}\right)^2 \left(d^{2}(x^{ks},C)-d^{2}(x^{(k+1)s},C) \right).
\end{equation}
Since (\ref{eq:th:rate:proof:5}) was proved for arbitrary $\nu \in I$, when we combine this with the bounded linear regularity of the family $\mathcal C$, we get
\begin{align}
d^{2}(x^{ks},C) &\leq \kappa_r^{2} \max_{i\in I} d^2(x^{k},C_{i}) \nonumber\\
&\leq  \frac {s} {\omega^-\rho^-(\alpha^-)^2} \left(\frac{\kappa_r\Delta_r }{\delta_r}\right)^2 \left(d^{2}(x^{ks},C)-d^{2}(x^{(k+1)s},C)\right),
\end{align}
which when rearranged, leads to
\begin{equation}
d^{2}(x^{(k+1)s},C) \leq
\left(1-
\frac{\omega^-\rho^-(\alpha^-)^2}{s}\cdot
\left(\frac{\delta_r}{\kappa_r\Delta_r }\right)^2
\right)d^{2}(x^{ks},C).
\end{equation}
By applying Theorem \ref{th:Fejer:weak} (iii), we see that the subsequence $\{x^{ks}\}_{k=0}^\infty$ converges linearly and
\begin{equation}
\| x^{ks}-x^\infty\|\leq 2 d(x^0,C) \left(1-
\frac{\omega^-\rho^-(\alpha^-)^2}{s}\cdot
\left(\frac{\delta_r}{\kappa_r\Delta_r }\right)^2
\right)^{\frac{k}{2}}.
\end{equation}
From  Theorem \ref{th:Fejer:weak} (iv), it follows that the entire sequence $\{x^k\}_{k=0}^\infty$ converges linearly and, in addition, (\ref{eq:th:rate:proof:estimate}) holds true.

\textbf{Step 3.} Next we show that estimate (\ref{eq:th:rate:estimateEB}) holds true. Indeed, by combining Theorem \ref{th:Fejer:weak} (v) with the bounded linear regularity of the family $\mathcal C$ and (\ref{eq:th:rate:dipiUi}), one observes that
\begin{align}\nonumber
\|x^k-x^\infty\|&\leq 2d(x^k,C)\leq 2 \kappa_r \max_{i\in I}d(x^k,C_i)\\
\nonumber
& \leq \frac{2\kappa_r}{\delta_r}\max_{i\in I}p_i(x^k) \leq \frac{2\Delta_r\kappa_r}{\delta_r}\max_{i\in I}\|U_ix^k-x^k\|\\
\nonumber
& \leq \frac{2\Delta_r\kappa_r}{\delta_r}\max_{i\in I}d(x^k,C_i) \leq \frac{2\Delta_r\kappa_r}{\delta_r} d(x^k,C)\\
& \leq \frac{2\Delta_r\kappa_r}{\delta_r} \|x^k-x^\infty\|,
\end{align}
which, when combined with (\ref{eq:th:rate:estimate}), completes the proof.
\end{proof}

\ac{
\begin{remark}
We show that condition (30) is more general than (46). For simplicity, assume that $U_i=P_{C_i}$, $p_i=d(\cdot,C_i)$, $J_k=I$, $I_{2k}=I$ and $I_{2k+1}$ is arbitrary for every $k=0,1,2,\ldots$. In particular, we can always choose the inner block such that $I_{2k+1}\cap \Argmax_{j\in I}p_j(x^{2k+1})=\emptyset$. Clearly, condition (46) does not hold. Now we claim that (30) does hold in this case. Indeed, assume that $\max_{i\in I_k}d(x^k,C_i)\rightarrow_k 0$. We have to show that $\max_{i\in I}d(x^k,C_i)\rightarrow_k 0$, which by the definition of $I_k$ and (35) is already true for even $k$'s. We have
\begin{align}
d(x^{2k+1}, C_i) &= \|x^{2k+1}-P_{C_i}x^{2k+1}\| \leq \|x^{2k+1}-P_{C_i}x^{2k}\|\\
&\leq \|x^{2k+1}-x^{2k}\|+\|x^{2k}-P_{C_i}x^{2k}\|
\end{align}
and the right-hand side of the above inequality converges to zero since $\|x^{k+1}-x^k\|\rightarrow 0$ by (37). Note that (35) and (37) do not depend on (30).
\end{remark}

Assume that $p_i(x)=\|U_ix-x\|$ for every $i\in I$. Clearly, in this case condition (31) is satisfied and condition (30) takes the following form:
\begin{equation}\label{eq:th:conv:BR:IkJk}
\lim_{k\rightarrow\infty}\max_{i\in I_k}\|U_ix^k-x^k\|=0 \quad\Longrightarrow\quad \lim_{k\rightarrow\infty}\max_{j\in J_k}\|U_jx^k-x^k\|=0.
\end{equation}
Consequently we can simplify the statement of Theorem 16 to achieve weak and strong convergence while using weakly and boundedly regular operators. Moreover, since we can take $\Delta_r=1$ in (47), we can replace this condition by the following inequality: $\delta_r d(x,\fix U_i)\leq \|U_ix^k-x^k\|$. Note that the existence of $\delta_r$ can be guaranteed by assuming that every $U_i$ is boundedly linearly regular and therefore we can obtain a special case of Theorem 17. Other particular variants of Theorems 16 and 17 can be obtained by considering projection and subgradient projection methods which we discuss in more detail in Sections 3.1 and 3.2.

Note that, in general, the parameters $\delta_r, \Delta_r$ and $\kappa_r$ are not known explicitly. Nevertheless, formula (48) still enables us to compare $q_r$'s related to different methods. We find this observation important since $q_r$ characterizes convergence speed in terms of iterations. In particular, by using (48) one can deduce which operations may reduce $q_r$ and as a consequence, which methods should be faster. Moreover, one can deduce a natural stopping rule from (49) in terms of maximum proximity. Note that the maximum proximity which appears in (49) is entirely computable in contrast with $c_r$ and $q_r$. In the following example we show how one can compute $q_r$.
}

\begin{example}\label{ex:FPA}
For simplicity, let $I=J_1\cup\ldots\cup J_s$ and let the outer control $\{J_{[k]}\}_{k=0}^\infty$ change cyclically with the constant size of the outer block, that is, $[k]:=(k \text{ mod } s)+1$ and $|J_{[k]}|=b$ for every $k=0,1,2,\ldots$ and some integer $b$. \ac{Moreover, assume that $m=s\cdot b$. Also, let }
$\alpha_k=1$ and $\omega_i^k=1/|I_k|$ for every $k=0,1,2,\ldots$. We now present several fixed point algorithms (FPA) which satisfy the inner-outer control relation defined in (\ref{eq:th:rate:IkJk}) and thus also in (\ref{eq:th:conv:pipj}). In addition, by assuming the existence of constants $\kappa_r, \delta_r$ and $\Delta_r$, and by Theorem \ref{th:rate}, we may deduce the upper bound for the coefficient $q_r$.
\begin{enumerate}[(a)]
  \item Cyclic FPA: $x^{k+1}:=U_{[k]}x^k,$ where $s=m$ and
  \begin{equation}\label{eq:qr:cyclic}
  q_r=\sqrt[2m]{1-\frac{1}{m}\left(\frac{\delta_r}{\Delta_r \kappa_r}\right)^2}.
  \end{equation}

  \item Simultaneous FPA: $x^{k+1}:=\frac 1 b\sum_{i\in J_{[k]}} U_ix^k$, where
  \begin{equation}\label{eq:qr:simultaneous}
  q_r=\sqrt[\frac{2m}{b}]{1-\frac{1}{m}\left(\frac{\delta_r}{\Delta_r \kappa_r}\right)^2}.
  \end{equation}

  \item Simultaneous FPA with active sets: $ x^{k+1}:=\frac 1 {|I_k|}\sum_{i\in I_k} U_ix^k,$ where $I_k:=\{i\in J_{[k]}\mid p_i(x^k)>0\}$. Since we do not control the number of violated constraints, the upper bound for $q_r$ is the same as in (\ref{eq:qr:simultaneous}).

  \item Maximum proximity FPA: $x^{k+1}:= U_{i_k}x^k$,  where $i_k= \argmax_{j\in J_{[k]}}p_j(x^k)$ and
  \begin{equation}\label{eq:qr:maxprox}
  q_r=\sqrt[\frac{2m}{b}]{1-\frac{b}{m}\left(\frac{\delta_r}{\Delta_r \kappa_r}\right)^2}.
  \end{equation}

  \item Simultaneous FPA with the inner block of a fixed size: $ x^{k+1}:=\frac 1 t \sum_{i\in I_k} U_ix^k,$ where $I_k:=\{t$ smallest indices from $J_{[k]}$ with the largest proximity $p_j(x^k) \}$. We have
  \begin{equation}\label{eq:qr:simultaneous:t}
  q_r=\sqrt[\frac{2m}{b}]{1-\frac{b}{mt}\left(\frac{\delta_r}{\Delta_r \kappa_r}\right)^2}.
  \end{equation}

  \item Simultaneous FPA with the inner block determined by the threshold: $x^{k+1}:=\frac 1 {|I_k|}\sum_{i\in I_k} U_ix^k,$ where $I_k:=\{i\in J_{[k]}\mid p_i(x^k)\geq t \max_{j\in J_{[k]}} p_j(x^k)\}$ and $t\in [0,1]$. Again, since we do not control the number of constraints for which the proximity is above the threshold, we can only estimate the upper bound for $q_r$ by using (\ref{eq:qr:simultaneous:t}) and taking $t$ there to equal $b$. This leads to the same formula as in (\ref{eq:qr:simultaneous}).

\end{enumerate}
\end{example}

Notice that for all of the above-mentioned methods one could set $s=1$. Thus $b=m$ and $J_{[k]}=I$. In particular, for the maximum proximity FPA we get
\begin{equation}\label{eq:qr:maxprox:basic}
  q_r=\sqrt{1-\left(\frac{\delta_r}{\Delta_r \kappa_r}\right)^2}.
  \end{equation}
As a matter of fact, this is the best estimate that can be derived in our setting by using formula (\ref{eq:th:rate:estimate}). Nevertheless, the case where $J_{[k]}\subseteq I$ seems to be more applicable. We would like to emphasize that in cases (d)--(f), the $q_r$ is a decreasing function of $b$. Therefore, increasing the block size $b$ should speed up the convergence in terms of the iteration $k$, but at the cost of computational time. Moreover, the $q_r$ from case (e) is an increasing function of $t$, which implies that by considering more information from the outer block we may, in fact, reduce the convergence speed. This, when combined with (\ref{eq:qr:maxprox}), suggests that the maximum proximity FPA should perform in the best way out of all variants described in (a)--(f). Another observation is that by manipulating the parameter $t$ with fixed $b$, in cases (e) and (f), we may approach either the simultaneous or maximum proximity FPA from cases (b) and (d). Finally, we conjecture at this point that the maximum proximity FPA has the same convergence properties with $b<m$ as it has in the case of $b=m$ \ac{in terms of iterations. }
All of this we verify numerically in the last section of our paper.

\begin{remark} \label{rem:Bauschke}
By definition, our double-layer fixed point algorithm requires cutters as the input operators. Nevertheless, this does not limit our generality. For example, following Bauschke, Noll and Phan \ac{\cite[Theorem 6.1]{BauschkeNollPhan2015}, }
we assume we want to use averaged operators $V_i$ in the framework of the one-layer ($I_k=J_k$) simultaneous fixed point algorithm. That is, we let the sequence $\{x^k\}_{k=0}^\infty$ be defined in the following way:
\begin{equation}\label{eq:rem:conv:xk}
x^0\in\mathcal H;\quad x^{k+1}=\sum_{i\in I_k}\lambda_i^kV_ix^k \text{ for } k=0,1,2,\ldots,
\end{equation}
where $I_k\subseteq I$ and \ac{$\lambda_i^k\in  (0,1)$ }
are such that $\sum_{i\in I_k}\lambda_i^k=1$. Notice that since for every $i\in I$, the operator $V_i$ is averaged, it can be written as $V_i=\id+\eta_i (W_i-\id)$, where \ac{$\eta_i\in (0,1)$ }
and $W_i$ is nonexpansive. Thus, by substituting
\begin{equation}
\alpha_k:=2\underset{i\in I}{\sum}\lambda_{i}^{k}\eta_i, \qquad
\omega_{i}^{k}:=\frac{\lambda_{i}^{k}\eta_i}{\underset{j\in I}{\sum}\lambda_{j}^{k}\eta_j}\qquad \text{and} \qquad
U_i:=\frac{1}{2}(W_i+\id)= \id+\frac{1}{2\eta_i}(V_i-\id)
\end{equation}
into (\ref{eq:th:conv:xk}), we recover (\ref{eq:rem:conv:xk}), where every $U_i$ is a cutter, \ac{$\alpha_k\in(0,2)$ }
and $\sum_{i\in I}\omega_i^k=1$. \ac{This also corresponds to \cite[Theorem 3.6]{BorweinLiTam2015}.}
\end{remark}

\begin{remark}
Let $\{x^k\}_{k=0}^\infty$ be Fej\'er monotone with respect to $C$. If $C$ is an affine subspace and all weak cluster points of $\{x^k\}_{k=0}^\infty$ belong to $C$, then, by \cite[Fact 5.3 (ii)]{BauschkeNollPhan2015}, $x^k\rightharpoonup x^\infty=P_Cx^0$. Consequently, if every $C_i$ in Theorem \ref{th:conv} is an affine subspace, then the sequence defined in (\ref{eq:th:conv:xk}) converges weakly to $x^\infty=P_Cx^0$. Moreover, this convergence turns out to be either strong or a linear one if $\mathcal C$ is either boundedly regular or boundedly linearly regular, respectively.
\end{remark}

\subsection{Projection algorithms}\label{sec:conv:proj}
Unless otherwise stated, from now on we assume that $C:=\bigcap_{i\in I}C_i \neq\emptyset$ and $C_i\subseteq\mathcal H$ is closed and convex for every $i\in I$.

\begin{corollary}[Weak and strong convergence]\label{th:conv:proj}
Let the sequence $\{x^k\}_{k=0}^\infty$ be defined by the double-layer  projection algorithm, that is,
\begin{equation}\label{eq:th:conv:proj:xk}
x^0\in\mathcal H;\quad x^{k+1}:=x^k+\alpha_k \left (\sum_{i\in I_k}  \ac{\omega_i^k}
 P_{C_i}x^k - x^k \right) \text{ for } k=0,1,2,\ldots,
\end{equation}
where $I_k\subseteq J_k\subseteq I$, $\alpha_k\in (0,2)$ and $\omega_i^k\in  (0,1]$ are such that $\sum_{i\in I_k}\omega_i^k=1$.
Assume that
\begin{enumerate}[(i)]
\item $\alpha_k\in[\alpha^-,\alpha^+]$ and $\omega^i_k\in[\omega^-, 1]$, where $\alpha^-, \omega^- \in (0,1]$ and $\alpha^+\in [1,2)$;

\item \ac{$I= J_k\cup\ldots\cup J_{k+s-1}$ }
for every $k=0,1,2,\ldots$ and some integer \ac{$s\geq 1$;}

\item the inner and outer controls satisfy
\begin{equation}\label{eq:th:conv:proj:pipj}
\lim_{k\rightarrow\infty}\max_{i\in I_k}d(x^k,C_i)=0 \quad\Longrightarrow\quad \lim_{k\rightarrow\infty}\max_{j\in J_k}d(x^k,C_j)=0;
\end{equation}

\end{enumerate}
Then the sequence $\{x^k\}_{k=0}^\infty$ converges weakly to some point $x^\infty \in C$. Moreover, if the family $\mathcal C:=\{C_i\mid i\in I\}$ is boundedly regular, then the convergence to $x^\infty$ is in norm.

\end{corollary}

\begin{proof}
This result follows from Theorem \ref{th:conv}, where it suffices to substitute $U_i=P_{C_i}$ and $p_i(x)=d(x,C_i)$.
\end{proof}

\begin{corollary}[Linear convergence]\label{th:rate:proj}
Let the sequence $\{x^k\}_{k=0}^\infty$ be defined as in \ac{Corollary \ref{th:conv:proj} }
under conditions (i) and (ii). Moreover, assume that

\begin{enumerate}[(i)]
\setcounter{enumi}{2}
\item the inner and outer controls satisfy, for each $k=0,1,2,\ldots$,
\begin{equation}\label{eq:th:rate:sub:IkJk}
I_k\cap \Argmax_{j\in J_k} d(x^k, C_j) \neq \emptyset;
\end{equation}
\end{enumerate}
If the family \ac{$\mathcal C:=\{C_i\mid i\in I\}$ }
is $\kappa_r$-boundedly linearly  regular over $B(P_Cx^0,r)$, then the sequence $\{x^k\}_{k=0}^\infty$ converges linearly to some point $x^\infty\in C$, that is, $\|x^k-x^\infty\|\leq c_rq_r^k$, where
\begin{equation}\label{eq:th:rate:sup:estimate}
c_r=\frac{2d(x^0,C)}
{ q_r^{s-1}}
\qquad\text{and}\qquad
q_r=\sqrt[\scriptstyle{2s}]{1-\frac{\omega^-(2-\alpha^+)(\alpha^-)^2}{s\alpha^+}\cdot
\frac{1}{\kappa_r^2}}.
\end{equation}
\emph{(Error bound)} Moreover, we have the following estimate
\begin{equation} \label{eq:th:rate:sup:estimateEB}
\frac{1}{2\kappa_r}\|x^k-x^\infty\|\leq \max_{i\in I} d(x^k,C_i)\leq c_r q_r^k.
\end{equation}
\end{corollary}

\begin{proof}
This result follows from Theorem \ref{th:rate} by substituting $U_i=P_{C_i}$ and $p_i(x)=d(x,C_i)$. Note that in this case $\delta_r= \Delta_r=1$ and $U_i$ is linearly regular, $i\in I$.
\end{proof}

\begin{example}[Example \ref{ex:FPA} revisited]\label{ex:PM}
We now present several projection methods (PM) which we obtain by substituting $U_i=P_{C_i}$ and $p_i=d(\cdot,C_i)$ in Example \ref{ex:FPA}. By assuming the existence of $\kappa_r>0$ we can apply Corollary \ref{th:rate:proj} to find the constant $q_r$.
\begin{enumerate}[(a)]
  \item Cyclic PM:
  $x^{k+1}:=P_{C_{[k]}}x^k$, where $s=m$ and
  \begin{equation}\label{eq:qr:cyclic:PM}
  q_r=\sqrt[2m]{1-\frac{1}{ m\kappa^2_r}}.
  \end{equation}
  A similar rate can be found, for example, in \cite[Theorem 1]{GurinPolyakRaik1967} by Gurin et al., in \cite[Theorem 5.7]{BauschkeBorwein1996} by Bauschke and Borwein, and in \cite[Theorem 4.5]{DeutschHundal2008} by Deutsch and Hundal.

  \item Simultaneous PM:
  $x^{k+1}:=\frac 1 b\sum_{i\in J_{[k]}} P_{C_i}x^k$, where
  \begin{equation}\label{eq:qr:simultaneous:PM}
  q_r=\sqrt[\frac{2m}{b}]{1-\frac{1}{m \kappa_r^2}}.
  \end{equation}
   A similar $q_r$ can be deduced from \cite[Theorem 2.2]{BeckTeboulle2003} by Beck and Teboulle with $b=m$. We recall that a linear convergence rate for $b=m$ appeared in \cite[Theorem 1.1]{Pierra1984} by Pierra.

  \item Simultaneous PM with active sets:
  $ x^{k+1}:=\frac 1 {|I_k|}\sum_{i\in I_k} P_{C_i}x^k,$ where $I_k:=\{i\in J_{[k]}\mid d(x^k,C_i)>0\}$. By similar arguments to those used in Example \ref{ex:FPA}, the upper bound for $q_r$ is the same as in (\ref{eq:qr:simultaneous:PM}).

  \item Remotest-set PM:
  $x^{k+1}:= P_{C_{i_k}}x^k$, where $i_k= \argmax_{j\in J_{[k]}}d(x^k,C_j)$ and
  \begin{equation}\label{eq:qr:maxprox:PM}
  q_r=\sqrt[\frac{2m}{b}]{1-\frac{b}{m \kappa_r^2}}.
  \end{equation}
  A similar estimate for $b=m$ can be found, for example, in \cite[Theorem 1]{GurinPolyakRaik1967} by Gurin et al. and in \cite[Theorem 5.8]{BauschkeBorwein1996} by Bauschke and Borwein. It can also be deduced from \cite[Theorem 2.2]{BeckTeboulle2003} by Beck and Teboulle. We note here that originally, for $b\leq m$, this PM appeared in \cite[Section 5.8.4]{Cegielski2012} by Cegielski.

  \item Simultaneous PM with the inner block of a fixed size:
  $ x^{k+1}:=\frac 1 t \sum_{i\in I_k} P_{C_i}x^k,$ where $I_k:=\{t$ smallest indices from $J_{[k]}$ with the largest distance $d(x^k,C_j) \}$. We have
  \begin{equation}\label{eq:qr:simultaneous:t:PM}
  q_r=\sqrt[\frac{2m}{b}]{1-\frac{b}{mt\kappa_r^2}}.
  \end{equation}

  \item Simultaneous PM with the inner block determined by the threshold:
  $x^{k+1}:=\frac 1 {|I_k|}\sum_{i\in I_k} P_{C_i}x^k,$ where $I_k:=\{i\in J_{[k]}\mid d(x^k,C_i)\geq t \max_{j\in J_{[k]}} d(x^k,C_j)\}$ and $t\in [0,1]$. Since we do not control the number of constraints for which the distance functional is above the threshold, we can only estimate the upper bound for $q_r$ by using (\ref{eq:qr:simultaneous:t:PM}) and setting there $t=b$. This leads us to formula (\ref{eq:qr:simultaneous:PM}). We recall here that originally, for $b\leq m$, this PM appeared in \cite[Section 5.8.4]{Cegielski2012} by Cegielski.
\end{enumerate}

\end{example}

\begin{remark}\label{rem:proxForLinSys}
Notice that for the projection methods designed for solving linear systems ($Ax=b$) and systems of linear inequalities ($Ax\leq b$), we can define the proximities
$p_i(x)=|\langle a_i,x\rangle-b_i|$ and $p_i(x)=(\langle a_i,x\rangle-b_i)_+$, respectively. By setting $U_i=P_{C_i}$  we see that $\delta_r=\min_i \|a_i\|$ and $\Delta_r=\max_i\|a_i\|$.
\end{remark}

\subsection{Subgradient projection algorithms}\label{sec:conv:sub}
Unless otherwise stated, from now on we assume that $C:=\bigcap_{i\in I}C_i \neq\emptyset$ and for every $i\in I$, we have
\begin{equation}
C_i=S(f_i,0)=\{z\in\mathcal H\mid f_i(z)\leq 0\}\subseteq\mathcal H,
\end{equation}
where $f_i\colon\mathcal H\rightarrow\mathbb R$ is convex and continuous.

\begin{corollary}[Weak convergence in $\mathcal H$]\label{th:conv:sub:H}
Let the sequence $\{x^k\}_{k=0}^\infty$ be defined by the double-layer subgradient projection algorithm, that is,
\begin{equation}\label{eq:th:conv:sub:H:xk}
x^0\in\mathcal H;\quad x^{k+1}:=x^k+\alpha_k \left (\sum_{i\in I_k}  \ac{\omega_i^k}
 P_{f_i}x^k - x^k \right) \text{ for } k=0,1,2,\ldots,
\end{equation}
where $I_k\subseteq J_k\subseteq I$, $\alpha_k\in (0,2)$ and $\omega_i^k\in  (0,1]$ are such that $\sum_{i\in I_k}\omega_i^k=1$.
Assume that
\begin{enumerate}[(i)]
\item $\alpha_k\in[\alpha^-,\alpha^+]$ and $\omega^i_k\in[\omega^-, 1]$, where $\alpha^-, \omega^- \in (0,1]$ and $\alpha^+\in [1,2)$;

\item \ac{$I= J_k\cup\ldots\cup J_{k+s-1}$ }
for every $k=0,1,2,\ldots$ and some integer \ac{$s\geq 1$;}

\item the inner and outer controls satisfy
\begin{equation}\label{eq:th:conv:sub:H:IkJk}
\lim_{k\rightarrow\infty}\max_{i\in I_k}\|P_{f_i}x^k-x^k\|=0 \quad\Longrightarrow\quad \lim_{k\rightarrow\infty}\max_{j\in J_k}\|P_{f_j}x^k-x^k\|=0;
\end{equation}
\item for each $i\in I$, the function $f_i$ is Lipschitz continuous on bounded sets.
\end{enumerate}
Then, by (iv), $P_{f_i}$ is weakly regular, $i\in I$. Consequently, the sequence $\{x^k\}_{k=0}^\infty$ converges weakly to some point $x^\infty \in C$.
\end{corollary}

\begin{proof}
The weak regularity of $P_{f_i}$ follows from (iv); compare with Example \ref{ex:subProj} (a). The rest follows from Theorem \ref{th:conv}, where it suffices to substitute $U_i=P_{f_i}$ and $p_i(x)=\|P_{f_i}x-x\|$.
\end{proof}

\begin{remark}
We emphasize that in view of Example \ref{ex:subProj} (a), one can replace condition (iv) either by: \textit{$f_i$ maps bounded sets onto bounded sets, $i\in I$}; or by: \textit{$\partial f_i$ is uniformly bounded on bounded sets, $i\in I$}.
\end{remark}

\begin{corollary}[Convergence in $\mathbb R^n$]\label{th:conv:sub:Rn}
Let the sequence $\{x^k\}_{k=0}^\infty$ be defined as in Corollary \ref{th:conv:sub:H} under conditions (i) and (ii). Moreover, assume that
\begin{enumerate}[(i)]
\setcounter{enumi}{2}
\item the inner control and outer control satisfy
\begin{equation}\label{eq:th:conv:sub:Rn:IkJk}
\lim_{k\rightarrow\infty}\max_{i\in I_k}f^+_i(x^k)=0 \quad\Longrightarrow\quad \lim_{k\rightarrow\infty}\max_{j\in J_k}f^+_j(x^k)=0;
\end{equation}
\item $\mathcal H=\mathbb R^n$.
\end{enumerate}
Then,  by (iv), for every bounded sequence $\{y^k\}_{k=0}^\infty\subseteq\mathbb R^n$ and $i\in I$,
\begin{equation}
\lim_{k\rightarrow\infty}f^+_i(y^k)= 0 \quad\Longleftrightarrow\quad \lim_{k\rightarrow\infty}\|P_{f_i}y^k-y^k\|= 0 \quad\Longleftrightarrow\quad
\lim_{k\rightarrow\infty}d(y^k,S(f_i,0))= 0.
\label{eq:th:conv:sub:Rn:fiPidi}
\end{equation}
In particular, this holds for $\{x^k\}_{k=0}^\infty$ defined in (\ref{eq:th:conv:sub:H:xk}). Consequently, the sequence $\{x^k\}_{k=0}^\infty$ converges to some point $x^\infty \in C$.
\end{corollary}

\begin{proof}
To apply Theorem \ref{th:conv} with $U_i=P_{f_i}$ and $p_i=f_i^+$, it suffices to show that condition (\ref{eq:th:conv:sub:Rn:fiPidi}) holds true.
Indeed, let $\{y^k\}_{k=0}^\infty\subseteq S$, where $S\subseteq\mathbb R^n$ is bounded. Since $\mathcal H=\mathbb R^n$, for each $i\in I$, the subdifferential $\partial f_i$ is uniformly bounded on bounded sets; see Example \ref{ex:subProj} (a). In particular, there is $\Delta_S$ such that for every $i\in I$ and $x\in S$, $\|g_{f_i}(x)\|\leq \Delta_S$. Consequently, for every $k=0,1,2,\ldots,$ we have
\begin{equation}
\|g_{f_i}(y^k)\|\leq \Delta_S.
\end{equation}
Therefore, by the definition of the subgradient projection, see (\ref{eq:def:subProj:short}),
\begin{equation}
f^+_i(y^k)\leq \Delta_S\|P_{f_i}y^k-y^k\|\leq \Delta_S d(y^k, S(f_i,0)),
\end{equation}
where the last inequality follows from the fact that $P_{f_i}$ is a cutter; see (\ref{eq:subProj:cutter}) and Remark \ref{rem:SQNEandMP}.

To complete the proof it suffices to show that $f^+_i(y^k)\rightarrow 0$ implies that $d(y^k,C_i)\rightarrow 0$. Indeed, following \cite[Lemma 24]{CegielskiZalas2013}, we assume to the contrary that $f^+_i(y^k)\rightarrow 0$ and $\limsup_k d(y^k,S(f_i,0))= \varepsilon>0$. Since $\{y^k\}_{k=0}^\infty$ is bounded, by eventually passing to a subsequence, without loss of generality we may assume that $y^k\rightarrow y$ and $\lim_k d(y^k,S(f_i,0))= \varepsilon$. By continuity of $d(\cdot,S(f_i,0))$, we have $d(y,S(f_i,0))=\varepsilon$. The lower semicontinuity of $f^+_i$ implies that $0=\lim_kf^+_i(y^k)=\liminf_k f^+_i(y^k)\geq f^+_i(y)$, which is in contradiction with $\varepsilon >0$.
\end{proof}

\begin{corollary}[Linear convergence in $\mathbb R^n$]\label{th:rate:sub:Rn}
Let the sequence $\{x^k\}_{k=0}^\infty$ be defined as in Corollary \ref{th:conv:sub:H} under conditions (i) and (ii). Moreover, assume that

\begin{enumerate}[(i)]
\setcounter{enumi}{2}
\item the inner and outer controls satisfy, for each $k=0,1,2,\ldots$,
\begin{equation}\label{eq:th:rate:sub:Rn:IkJk}
I_k\cap \Argmax_{j\in J_k} f^+_j(x^k) \neq \emptyset;
\end{equation}

\item $\mathcal H=\mathbb R^n$ and for each $i\in I$, \ac{there is a point $z^i$ such that $f_i( z^i)<0$.}
\end{enumerate}
Then, by (iv), for $r=d(x^0,C)>0$, there are numbers $\delta_r, \Delta_r>0$ such that for every $x\in B(P_Cx^0,r)$ and $i\in I$,
\begin{equation}\label{eq:th:rate:sup:Rn:difiUi}
\delta_r d(x,S(f_i,0))\leq f^+_i(x) \leq \Delta_r \|P_{f_i}x-x\|.
\end{equation}
In addition, if the family $\mathcal C:=\{S(f_i,0)\mid i\in I\}$ is $\kappa_r$-boundedly linearly regular over $B(P_Cx^0,r)$, in particular, \ac{when $\max_{i\in I}f_i(z)<0$ for some $z$ (Slater condition), }
then the sequence $\{x^k\}_{k=0}^\infty$ converges linearly to some point $x^\infty\in C$, that is, $\|x^k-x^\infty\|\leq c_rq_r^k$, where
\begin{equation}\label{eq:th:rate:sup:Rn:estimate}
c_r=\frac{2d(x^0,C)}
{ q_r^{s-1}}
\qquad\text{and}\qquad
q_r=\sqrt[\scriptstyle{2s}]{1-\frac{\omega^-(2-\alpha^+)(\alpha^-)^2}{s\alpha^+}\cdot
\left(\frac {\delta_r }{\Delta_r \kappa_r}\right)^2}.
\end{equation}
\emph{(Error bound)} Moreover, we have the following estimate:
\begin{equation} \label{eq:th:rate:Rn:estimateEB}
\frac{\delta_r}{2\kappa_r}\|x^k-x^\infty\|\leq \max_{i\in I} f^+_i(x^k)\leq \Delta_r c_r q_r^k.
\end{equation}
\end{corollary}

\begin{proof}
It suffices to show that condition (\ref{eq:th:rate:sup:Rn:difiUi}) holds which, by Theorem \ref{th:rate} with $U_i=P_{f_i}$ and $p_i=f_i^+$, yields the result.
Indeed, since $\mathcal H=\mathbb R^n$, for each $i\in I$, the subdifferential $\partial f_i$ is uniformly bounded on bounded sets; see Example \ref{ex:subProj} (a). This holds, in particular, for the ball $B(P_Cx^0, r=d(x^0,C))$. Therefore there is a number $\Delta_r$ such that for every $i\in I$ and $x\in B(P_Cx^0, r)$, we have $\|g_{f_i}(x)\|\leq \Delta_r$. Thus, by the definition of the subgradient projection, see (\ref{eq:def:subProj:short}), we get
\begin{equation}
f^+_i(x) \leq \Delta_r \|P_{f_i}x-x\|.
\end{equation}

Now, assume that for each $i\in I$, there is $ z^i\in S(f_i,0)$ such that $f_i(z^i)<0$ (Condition (iv)). Then, by Lemma \ref{th:Fukushima} applied to eaxch $f_i$ separately, for every compact set $K$ there is $\delta_K>0$ such that for every $x\in K$, $\delta_Kd(x,S(f_i,0))\leq f_i^+(x)$. This holds, in particular, for $K=B(P_Cx^0,r)$ with a common constant $\delta_r$, which completes the proof.
\end{proof}

\begin{example}[Examples \ref{ex:FPA} and \ref{ex:PM} continued] \label{ex:SPM}
We now present several subgradient projection methods (SPM), which we obtain by substituting $U_i=P_{f_i}$ and $p_i=f_i^+$ in Example \ref{ex:FPA}. By assuming the existence of $\kappa_r>0$ and condition (iv) from Corollary \ref{th:rate:proj}, we can obtain the constant $q_r$.
\begin{enumerate}[(a)]
  \item Cyclic SPM:
  $x^{k+1}:=P_{f_{[k]}}x^k$ with $q_r$ defined in (\ref{eq:qr:cyclic}). A similar rate can be found, for example, in \cite[Theorem 2]{DePierroIusem1988} by De Pierro and Iusem.

  \item Simultaneous SPM:
  $x^{k+1}:=\frac 1 b\sum_{i\in J_{[k]}} P_{f_i}x^k$ with $q_r$ defined in (\ref{eq:qr:simultaneous}).

  \item Simultaneous SPM with active sets:
  $ x^{k+1}:=\frac 1 {|I_k|}\sum_{i\in I_k} P_{f_i}x^k,$ where $I_k:=\{i\in J_{[k]}\mid f_i^+(x^k)>0\}$ where $q_r$ is defined in (\ref{eq:qr:simultaneous}).

  \item SPM with the most-violated constraint, that is:
  $x^{k+1}:= P_{f_{i_k}}x^k$, where $i_k= \argmax_{j\in J_{[k]}} f_j^+(x^k)$ and $q_r$ is defined by formula (\ref{eq:qr:maxprox}). We recall that for $b=m$, the linear convergence rate can be found, for example, in  \cite[Theorem, p. 142]{Eremin1968}  by Eremin and in \cite[Theorem 6]{Polyak1969} by  Polyak.

  \item Simultaneous SPM with the inner block of a fixed size:
  $ x^{k+1}:=\frac 1 t \sum_{i\in I_k} P_{f_i}x^k,$ where $I_k:=\{t$ smallest indices from $J_{[k]}$ with the largest violation $f_j^+(x^k) \}$, where $q_r$ is defined in (\ref{eq:qr:simultaneous:t}).

  \item Simultaneous SPM with the inner block determined by the threshold:
  $x^{k+1}:=\frac 1 {|I_k|}\sum_{i\in I_k} P_{f_i}x^k,$ where $I_k:=\{i\in J_{[k]}\mid f_i^+(x^k)\geq t \max_{j\in J_{[k]}} f_j^+(x^k)\}$ and $t\in [0,1]$. Here $q_r$ is defined in (\ref{eq:qr:simultaneous}).

\end{enumerate}

\end{example}

\ac{
\subsection{Lopping and flagging}\label{sec:conv:loppFlagg}
Note that the simple choice of the outer control $\{J_k\}_{k=0}^\infty$ presented in Example 18 can easily be extended by applying the so-called \textit{lopping} and \textit{flagging} methodology \cite{Haltmeier2009, ElfvingHansenNikazad2016}. Both of these techniques were brought to our attention by an anonymous referee. This, in particular, can be considered a practical realization of the $s$-intermittent control; see (3). We recall that the idea of lopping presented in \cite{Haltmeier2009, ElfvingHansenNikazad2016} is to omit computation of algorithmic operators assigned to a block if that block has no active constraints and proceed to another block. One checks if the block has active constraints by comparing a certain violation measure with a user chosen threshold. By extending this idea one can flag the nonactive block as not available for the upcoming $N$ iterations and unflag it afterwards, so that it becomes available again. To this end, assume that $I=J_1'\cup\ldots \cup J_s'$ for some $s\geq 1$, $J_k'=J_{[k]}'$ for all $k\geq s+1$, where $[k]=(k\text{ mod } s)+1$. Moreover, let $N\geq 1$ and $\varepsilon \geq 0$ be user chosen parameters. We will use $\varepsilon$ for flagging blocks as well as for the stopping rule. Following \cite{ElfvingHansenNikazad2016}, we propose the following algorithm:

\begin{algorithm2}[Double layer control with lopping and flagging]  \rm~
\begin{itemize}
    \item[] \textbf{Initialize.} Set $n=k:=0$, $J_0=J_0'=:J_1'$ and flag all the blocks $\{J_k'\}_{k=0}^\infty$ as available. Moreover, choose $x^0\in\mathcal H$ and mark $J_0'$ as the last used block.
    \item[] \textbf{Step 1.} For a given $k\geq 0$ and $x^k$, let $J_l'$ be the last used block for some $l\geq k$. Take the next available block after $J_l'$ and denote it by $J_t'$. Set $J_k:=J_t'$.
    \item[] \textbf{Step 2.} If $\max_{j\in J_k}p_j(x^k)>\varepsilon$, then set $n:=0$, select a proper $I_k$ from $J_k$ and compute $x^{k+1}$ using formula (29). If $\max_{j\in J_k}p_j(x^k)\leq\varepsilon$, then set $n:=n+1$, $x^{k+1}:=x^k$ and flag blocks $J_{t}', J_{t+s}',\ldots,J_{t+Ns}'$ as not available.
    \item[] \textbf{Step 3.} If $n<s$ then set $k:=k+1$ and return to step 1. If $n=s$ then stop the algorithm since $\max_{i\in I }p_i(x^k)\leq \varepsilon$.
\end{itemize}
\end{algorithm2}

It is easy to see that every basic block $J_l'$, $l=1,\ldots,s$, has to appear at least once in $\{J_k,\ldots, J_{k+Ns-1}\}$ and therefore $\{J_k\}$ is $(Ns)$-intermittent. Therefore, when $\varepsilon=0$, the sequence $\{x^k\}_{k=0}^\infty$ defined above may converge either weakly, strongly or even linearly by Theorems 16 and 17. When $\varepsilon>0$, then one can show that in a finite number of iterations we can reach the stopping rule $\max_{i\in I}p_i(x^k)<\varepsilon$ under the assumption that the inner and outer controls satisfy (46), and that $p_i$ and $U_i$ are related as in (31). The argument is a slight modification of the proof of Theorem 16. We sketch it below for the convenience of the reader.

Indeed, let $K:=\{k\mid \max_{j\in J_k}p_i(x^k)>\varepsilon\}$. Note that $x^{k+1}\neq x^k$ may happen only when $k\in K$. If $K$ is finite then the algorithm stops. Assume that $K$ is infinite. For every $k\in K$ and $z\in C$, we have
\begin{equation}
\|x^{k+1}-z\|^2=\|T_kx^k-z\|^2\leq \|x^k-z\|^2 - \rho\|T_kx^k-x^k\|,
\end{equation}
where $T_k$ and $\rho$ are defined as in (32)--(33). Therefore, for every $k=0,1,2,\ldots$, we have $\|x^{k+1}-z\|\leq \|x^k-z\|$. This implies that the sequence $\{x^k\}_{k=0}^\infty$ is Fej\'{e}r monotone with respect to $C$, $\{\|x^k-z\|\}_{k=0}^\infty$ converges and consequently,
\begin{equation}
\lim_{k\rightarrow\infty \atop k\in K}\|T_kx^k-x^k\|=0.
\end{equation}
In view of (38), which holds for every $k\in K$ by (31) and (46), we get
\begin{equation}
\lim_{k\rightarrow\infty \atop k\in K} \max_{j\in J_k}p_j(x^k)=0.
\end{equation}
This implies that that there is $k_0$ such that for every $k\geq k_0$, $\max_{j\in J_k}p_j(x^k)\leq \varepsilon$. Note that for $k\notin K$ this follows from the definition of the sequence $\{x^k\}_{k=0}^\infty$. Consequently, $K$ needs to be finite.
}

\section{Numerical experiments}\label{sec:num}

In this section we test the performance of selected algorithms on systems of linear inequalities $Ax\leq b$, where the dimensions of the matrix $A$ are $100 \times 20$. The linear inequalities are randomly generated so that they form a consistent convex feasibility problem.

Since for every $i\in I$, the $i$-th inequality corresponds to a half-space $C_i=\{z\in\mathbb R^{20}\mid\langle a_i,z\rangle \leq b_i\}$, following Remark \ref{rem:proxForLinSys}, we define
\begin{equation}
p_i(x):=(\langle a_i,x\rangle-b_i)_+; \qquad U_ix:=P_{C_i}x=x-\frac{p_i(x)}{\|a_i\|^2}a_i.
\end{equation}
To design our algorithms, we repeat the setting of Example \ref{ex:FPA}. First, we divide $I=\{1,\ldots,100\}$ into $s$ equal blocks of size $b$, that is, $I=J_1\cup\ldots\cup J_s$. For the last block we allow $|J_s|\leq b$. Second, we apply a cyclic outer control $\{J_{[k]}\}_{k=0}^\infty$, where $[k]:=(k \text{ mod } s)+1$. Moreover, we set $\alpha_k=1$ and $\omega_i^k=1/|I_k|$ for every $k=0,1,2,\ldots$.

The algorithms in which we are interested are the following projection methods:
\begin{itemize}
  \item Cyclic PM:
  $x^{k+1}:=P_{C_{[k]}}x^k$, where in this case $s=m$ and thus $[k]:=(k \text{ mod } m)+1$;

  \item Simultaneous PM:
  $x^{k+1}:=\frac 1 {| J_{[k]}|}\sum_{i\in J_{[k]}} P_{C_i}x^k$;

  \item Maximum proximity PM:
  $x^{k+1}:= P_{C_{i_k}}x^k$, where $i_k= \argmax_{j\in J_{[k]}}p_j(x^k)$;

  \item Simultaneous PM with the inner block of a fixed size:
  $ x^{k+1}:=\frac 1 t \sum_{i\in I_k} P_{C_i}x^k,$ where $I_k:=\{t$ smallest indices from $J_{[k]}$ with the largest proximity $p_j(x^k) \}$.

  \item Simultaneous PM with the inner block determined by a threshold:
  $x^{k+1}:=\frac 1 {|I_k|}\sum_{i\in I_k} P_{C_i}x^k,$ where $I_k:=\{i\in J_{[k]}\mid p_i(x^k)\geq t \max_{j\in J_{[k]}} p_j(x^k)\}$ and $t\in [0,1]$.
\end{itemize}
Note that in view of Remark \ref{rem:proxForLinSys}, for each one of the above-mentioned PM's we can apply the formulae of Example \ref{ex:FPA} with $\delta_r=\min_i \|a_i\|$ and $\Delta_r=\max_i\|a_i\|$.

For every algorithm we perform $N=100$ simulations, while sharing the same set of test problems. We apply two types of stopping rules: the first one is $\max_{i\in I} p_i(x^k)\leq 10^{-6}$, which we verify per every 100 iterations (first matrix dimension); the second one is an upper bound of 5000 iterations. After running all of the simulations, for every iterate we compute the maximum proximity. In order to compare our algorithms, we consider the quantity
\begin{equation}\label{eq:plotMaxProx}
\log_{10}\left(\frac{\max_{i\in I} p_i(x^k)}{\max_{i\in I} p_i(x^0)}\right).
\end{equation}
The bold line in Figures \ref{fig:1}-\ref{fig:4} indicates the median computed for (\ref{eq:plotMaxProx}). The ribbon plot represents the concentration of order 10, 20, 30, 40 and 50\% around the median.

We present now several observations which we have noticed after running the numerical simulations. These observations confirmed some of the predictions which followed Example \ref{ex:FPA}.
\begin{enumerate}[(a)]
  \item The maximum proximity PM with the largest possible outer block $J_{[k]}=I$   outperformed every other PM we considered, while simultaneous PM seems to be the slowest one; see Figures \ref{fig:1}--\ref{fig:3}.

  \item By reducing the block size $b$ in the maximum proximity PM, we approach the convergence rate of the cyclic PM. By increasing the size of the block, we speed up the convergence \ac{in terms of iterations. }
      Even for the small blocks, where $b=2,3,5$, the improvement is significant in comparison with the cyclic method. Moreover, in the case of $b=25$, the performance is almost the same as in the case of $b=100$; see Figure \ref{fig:1}. Therefore, in general, one could expect that the maximum proximity algorithms may have the same convergence properties with some $b<m$ as they have in the case of $b=m$.

  \item In Figure \ref{fig:2} we regulate the size of the inner block by adjusting the parameter $t$. Following (b), we have fixed the size of the outer block to be 25 ($b=25$). By reducing the parameter $t$ we approach the maximum proximity PM, while by increasing $t$ we approach the simultaneous PM. Surprisingly, by taking into account more information from the outer block we reduce, in fact, the convergence speed.

  \item In Figure \ref{fig:3} we control the size of the inner block by modifying the threshold determined by $t$. As above, we have fixed the size of the outer block to be 25. Similarly, by increasing the threshold we approach the maximum proximity PM and thus we speed up the convergence. On the other hand, by decreasing $t$ we approach the simultaneous PM which consequently slows down the convergence. As in case (c), by considering more information from the outer block once again, we reduce the convergence speed.

  \item Observations (c) and (d), as well as formula (\ref{eq:qr:simultaneous:t}), led us to the \ac{conclusion }
      that the smaller the ratio between the sizes of the inner and outer blocks is, the faster convergence we can expect. We show this effect in Figure \ref{fig:4}, where the ratios $t/b$ are $0.3=3/10=6/20=15/50$, $0.5=5/10=10/20=25/50$ and $0.7=7/10=14/20=35/50$.
\end{enumerate}

\begin{figure}[htb]
\centering
\includegraphics[bb=0.0 0.0 792.0 612.0, scale=0.4]{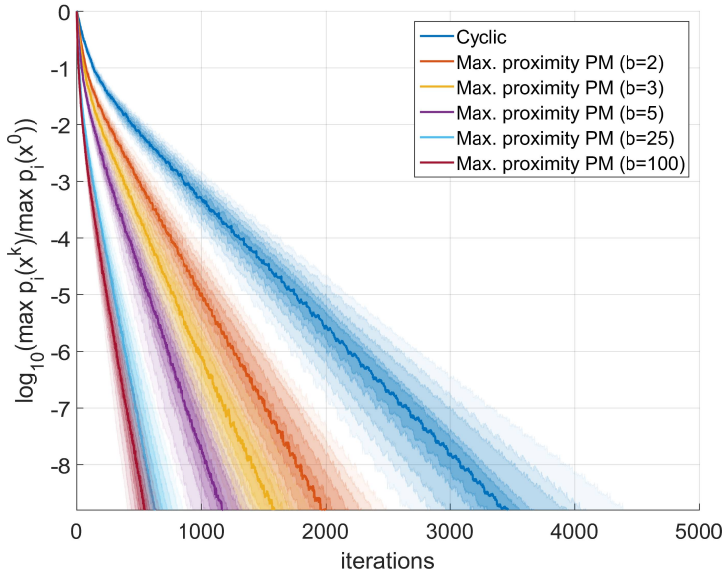}
\caption{Maximum proximity projection methods over the cyclic outer block $J_{[k]}$ of size $b=2,3,5,10$ and $25$. For the cyclic algorithm, $b=1$. A bold line indicates the median computed for (\ref{eq:plotMaxProx}). The ribbon plot represents the concentration of order 10, 20, 30, 40 and 50\% around the median.}
\label{fig:1}
\end{figure}

\begin{figure}[htb]
\centering
\includegraphics[bb=0.0 0.0 792.0 612.0, scale=0.4]{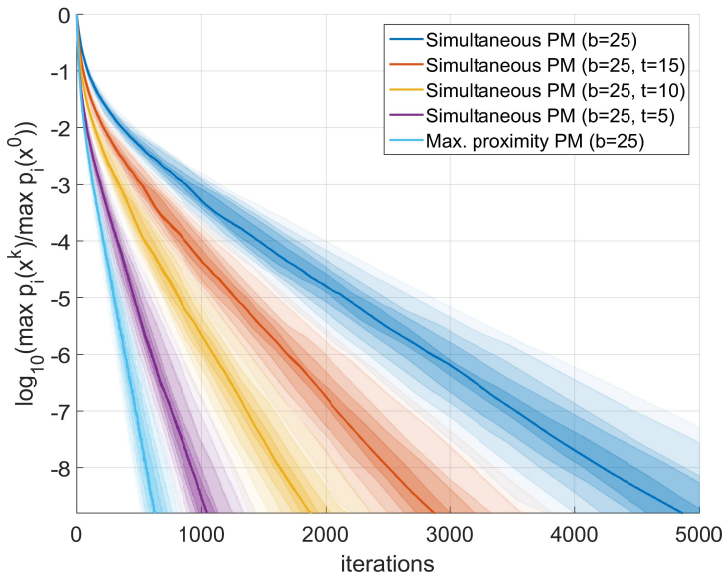}
\caption{Simultaneous projection methods over the cyclic outer block $J_{[k]}$ of size $b=25$. Indices inside the inner block $I_k$ are the first $t$ smallest indices with the largest proximity, where $t=15, 10$ and $5$. For the simultaneous method, $t=25$, whereas for the maximum proximity algorithm, $t=1$. Bold lines and ribbons are the same as in Figure \ref{fig:1}.}
\label{fig:2}
\end{figure}

\begin{figure}[htb]
\centering
\includegraphics[bb=0.0 0.0 792.0 612.0, scale=0.4]{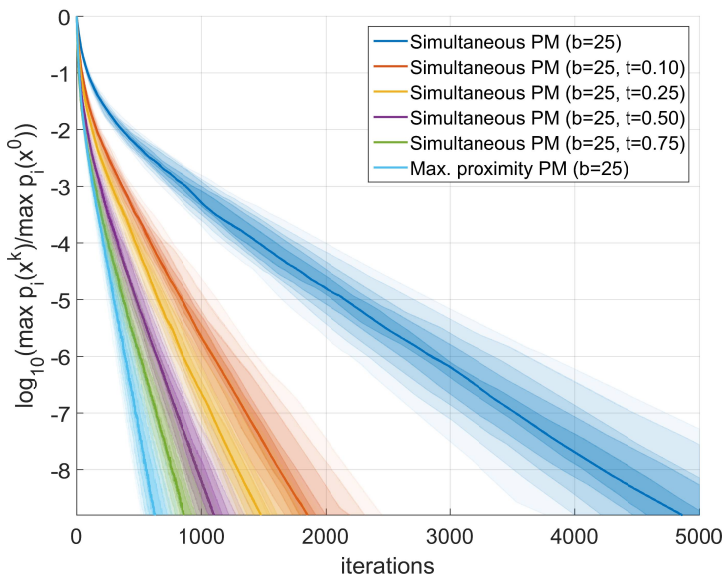}
\caption{Simultaneous projection methods over the cyclic outer block $J_{[k]}$ of the size $b=25$. Indices inside the inner block $I_k$ are those for which $p_i(x^k)\geq t \max_{j\in J_k}p_j(x^k)$, where $t=0.1,\ 0.25,\ 0.5$ and $0.75$. For the pure simultaneous method, $t=0$, whereas for the maximum proximity variant, $t$ could be considered to be $1$, while assuming that the maximum is uniquely attained. Bold lines and ribbons are the same as in Figure \ref{fig:1}.}
\label{fig:3}
\end{figure}

\begin{figure}[htb]
\centering
\includegraphics[bb=0.0 0.0 792.0 612.0, scale=0.4]{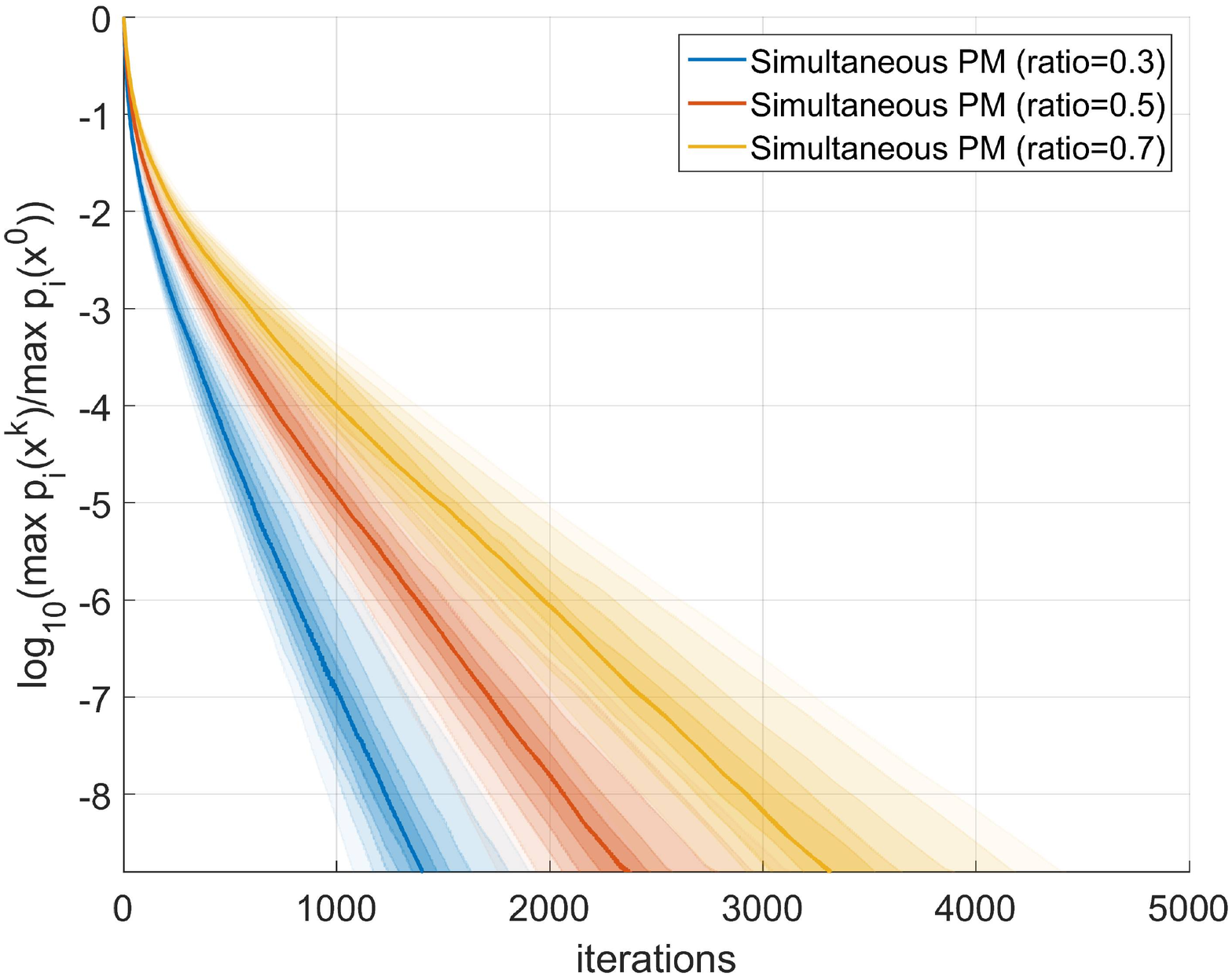}
\caption{Simultaneous projection methods over the cyclic outer block $J_{[k]}$, where the ratios $t/b$ between the sizes of the inner ($t$) and outer ($b$) blocks are $0.3=3/10=6/20=15/50$, $0.5=5/10=10/20=25/50$ and $0.7=7/10=14/20=35/50$. Bold lines and ribbons are the same as in Figure \ref{fig:1}.}
\label{fig:4}
\end{figure}

\section*{Acknowledgments} We are grateful to two anonymous referees for all their comments
and remarks which helped us to improve our manuscript.

\bibliographystyle{siamplain}
\bibliography{references}
\end{document}